\documentclass{article}

\usepackage[utf8]{inputenc}
\usepackage{amsmath,amsthm,amssymb,amsfonts}
\usepackage{graphicx}
\usepackage{xcolor}
\usepackage{enumitem}
\usepackage[unicode,colorlinks=true,pagebackref=false]{hyperref}

\usepackage{algorithm,algpseudocode}

\usepackage[
    textheight=9in,
    textwidth=6.5in,
    top=1in,
    headheight=14pt,
    headsep=25pt,
    footskip=30pt]{geometry}

\newtheorem{theorem}{Theorem}
\newtheorem{remark}[theorem]{Remark}

\newtheorem{corollary}[theorem]{Corollary}
\newtheorem{proposition}[theorem]{Proposition}

\numberwithin{equation}{section}
\numberwithin{theorem}{section}

\def\algorithmname{Algorithm}

\begin{document}

\title{Computing the unitary best approximant to the exponential function}
\author{Tobias Jawecki\thanks{tobias.jawecki@gmail.com}}
\maketitle
\begin{abstract}
Unitary best approximation to the exponential function on an interval on the imaginary axis has been introduced recently. In the present work two algorithms are considered to compute this best approximant: an algorithm based on rational interpolation in successively corrected interpolation nodes and the AAA-Lawson method. Moreover, a~posteriori bounds are introduced to evaluate the quality of a computed approximant and to show convergence to the unitary best approximant in practice. Two a~priori estimates---one based on experimental data, and one based on an asymptotic error estimate---are introduced to determine the underlying frequency for which the unitary best approximant achieves a given accuracy. Performance of algorithms and estimates is verified by numerical experiments. In particular, the interpolation-based algorithm converges to the unitary best approximant within a small number of iterations in practice.
\end{abstract}

\section{Introduction and overview}\label{sec:intro}
We consider rational approximation to the exponential function on an interval on the imaginary axis, i.e., for a given {\em frequency} $\omega>0$,
\begin{equation*}
r(\mathrm{i} x)\approx \mathrm{e}^{\mathrm{i} \omega x},~~~x\in[-1,1].
\end{equation*}
Unitary best approximation for this problem was introduced recently in~\cite{JS23u}.
We refer to rational functions $r=p/q$ where $p$ and $q$ have degree $\leq n$ as $(n,n)$-rational functions.
Moreover, we refer to a rational function $r$ as unitary if
\begin{equation*}
|r(\mathrm{i} x)|=1~~~\text{for $x\in\mathbb{R}$},
\end{equation*}
and we let $\mathcal{U}_n$ denote the class of unitary $(n,n)$-rational functions, i.e.,
\begin{equation*}
\mathcal{U}_n := \{p/q ~|~ \text{$p,q$ are polynomials of degree $\leq n$, $q\neq 0$, and $p/q$ is unitary}\}.
\end{equation*}
For a complex function $f$ we make use of the notation
\begin{equation}\label{eq:defnorm}
\|f\| := \max_{x\in[-1,1]}|f(\mathrm{i} x)|.
\end{equation}
In line with~\cite{JS23u}, for a given degree $n$ and frequency $\omega$ we refer to $\widetilde{r}\in\mathcal{U}_n$ as a unitary best approximant if
\begin{equation}\label{eq:uniformerror}
\|\widetilde{r} -  \exp(\omega \cdot)\|
= \min_{r\in\mathcal{U}_n} \|r - \exp(\omega \cdot)\|.
\end{equation}
We refer to a $(n,n)$-rational function $r$ as {\em degenerate} if $r$ can be equivalently described as a $(n-k,n-k)$-rational function for some $k>0$, and we refer to $r$ as non-degenerate otherwise.
Following~\cite{JS23u}, unitary best approximants exist for $\omega>0$ and are unique and non-degenerate for $\omega\in(0,(n+1)\pi)$. We assume the case $\omega\in(0,(n+1)\pi)$ throughout the present work.

\begin{remark}
While the unitary best approximation can be most accurately understood as an approximation $r(z)\approx \mathrm{e}^{\omega z}$ for $z\in\mathrm{i} [-1,1]$, we also refer to $r$ as an approximation to $\mathrm{e}^{\mathrm{i} \omega x}$ in an equivalent manner, i.e., $r(\mathrm{i} x)\approx \mathrm{e}^{\mathrm{i} \omega x}$ for $x\in[-1,1]$.
\end{remark}

\begin{remark}\label{rmk:approxww}
By rescaling function arguments, the unitary best approximation can equivalently be formulated as an approximation $r(z)\approx \mathrm{e}^{z}$ for $z\in\mathrm{i}[-\omega,\omega]$.
In particular, the rational function $\widehat{r}(z) := \widetilde{r}(z/\omega)$ satisfies
\begin{equation*}
\max_{x\in[-\omega,\omega]} |\widehat{r}(\mathrm{i} x) - \mathrm{e}^{\mathrm{i} x}|
= \|\widetilde{r}-\exp(\omega \cdot)\|.
\end{equation*}
\end{remark}

\subsection{Algorithms}
In the present work we consider the following two approaches to computing the unitary best approximant:
\begin{enumerate}[label={(\roman*)}]
\item\label{item:algbrib} 
Rational interpolation to $\mathrm{e}^{\mathrm{i} \omega x}$ in successively corrected interpolation nodes, and
\item\label{item:alglawson}
the AAA-Lawson method~\cite{NST18,NT20}.
\end{enumerate}
The approach~\ref{item:algbrib} is motivated by the ``Second Direct Method'' in~\cite{Ma63} and the BRASIL algorithm~\cite{Ho21} which utilize similar ideas to compute rational best approximants in a real setting. The former is also referred to as Maehly's second method in the literature~\cite{CS63,Du65,Du66,LR73}, and a modified version is introduced in~\cite{Fra76}.
For the approach~\ref{item:alglawson} we consider the AAA-Lawson method~\cite{NT20} where AAA stands for ``adaptive Antoulas--Anderson''. The first phase of the AAA-Lawson consists of the AAA method~\cite{NST18} which constructs a rational approximant in a greedy manner by rational interpolation in nodes adaptively selected from a given set of {\it test nodes}. The second phase of the AAA-Lawson method consists of a Lawson-type iteration in which accuracy is further improved by minimizing a successively reweighted least-squares problem without prescribing interpolation conditions. For our numerical experiments we consider equispaced test nodes and test nodes which are further adjusted on the run~\cite{DNT24}. For comparison purposes we also provide results for the AAA method without Lawson-type iterations, which aims to construct near-best approximants.

It was recently shown in~\cite{JS24} that rational interpolants to the exponential function on the imaginary axis (as utilized in~\ref{item:algbrib}) and rational approximants to $\mathrm{e}^{\mathrm{i} \omega x}$  generated by the AAA and AAA-Lawson methods for real-valued test nodes (as utilized in~\ref{item:alglawson}) are unitary. Thus, these approaches provide candidates for the unitary best approximant.

Other approaches to compute best approximants which are not considered in the present work are the \texttt{rkfit} method~\cite{BG15,BG17b} which aims to compute best approximants which minimize a least-squares error, algorithms in~\cite{EW76,IT93} which aim to compute complex Chebyshev approximants which are distinct to the unitary best approximant~\cite{Ja24ur}, and rational variants of the Remez method~\cite{PT09,FNTB18} which aim to compute Chebyshev approximants in a real setting.

\medskip
\noindent{\bf Software implementation.}
A Python implementation of the interpolation-based algorithm is provided by the authors in a Github repository\footnote{\url{https://github.com/newbisi/rexpi/}}. 
Matlab implementations of the AAA and AAA-Lawson methods are available as part of the Chebfun package~\cite{DHT14}.

\subsection{Aim of the present work}

The focus of the present work is on computing the unitary best approximant and ensuring the quality of the results in practice.
In particular, the interpolation-based algorithm that combines strategies of Maehly's second method and the BRASIL algorithm, as described in Subsection~\ref{subsec:bestalgorithm}, accomplishes this task for a wide range of degrees $n$ and frequencies $\omega$.
Comparison between different algorithms is not carried out in full detail, and the performance of the AAA-Lawson method to approximate $\mathrm{e}^{\mathrm{i} \omega x}$ might be further optimized in the future by adjusting the underlying set of test nodes.

We show that the uniform error of the unitary best approximant is sandwiched between the uniform error of a rational approximant with and without a scaling factor based on an error in uniformity, assuming an interpolatory setting. The error in uniformity provides an a~posteriori bound for a relative distance between the errors of a computed approximant and the unitary best approximant. We further show that an approximant converges to the unitary best approximant if the error in uniformity approaches zero. These results are utilized as stopping criteria for algorithms of type~\ref{item:algbrib} and to verify convergence of algorithms.

We also consider a~priori estimates for $\omega$ w.r.t\ $n$ s.t.\ the computed unitary best approximant attains a given error level. Such estimates have relevance for applications when approximants of a certain accuracy are required, and help to avoid settings where computing the unitary best approximant might fail in practice. The latter might be the case if the underlying unitary best approximant is nearly singular (the case $\omega\to (n+1)\pi$), or the computed approximant does not show equioscillation properties due to limitations of computer arithmetic (the case $\omega\to 0$).

\subsection{Applications and relevance}

Rational approximation to the exponential function has some relevance for approximation to the matrix exponential and time integration of systems of differential equations~\cite{ML03}. In this context, the approximation to the exponential function on the imaginary axis can be used to approximate the exponential operator of skew-Hermitian matrices whose eigenvalues lie on the imaginary axis. Such problems arise in particular for time-dependent Schr{\"o}dinger equations~\cite{Lu08}. It is shown in~\cite{JS23u} that unitary best approximants satisfy various properties that make them suitable for geometric numerical integration~\cite{HLW06}. While unitarity is typically desirable for applications, it can be understood as a restriction in terms of accuracy of best approximations. However, it is shown in \cite{Ja24ur} that the restriction to unitarity for rational best approximation to the exponential function on the imaginary axis is not severe in this context. Furthermore, the unitary best approximation, since it satisfies stability properties~\cite{JS23u}, is also suitable for slightly dissipative problems. 

Properties of the unitary best approximation are also compared with properties of polynomial Chebyshev approximation~\cite{TK84} (cf.~\cite[Subsection~III.2.1]{Lu08} for an overview) and the Pad\'e approximation~\cite{BG96}, both widely used in the context of Schr{\"o}dinger equations, in the introductory sections of~\cite{JS23u,JS24}. In particular, the unitary best approximation, which combines uniform accuracy of Chebyshev approximations with geometric properties of the Pad\'e approximation, shows great potential for use in future time integrators.

As a novelty of the present work, the interpolation-based algorithm for computing the unitary best approximant, which was partly also mentioned in \cite[Subsection~2.1.5]{JS23u}, is provided in full detail, performance of algorithms is tested and discussed, and convergence results and a~priori estimates for the underlying frequency are introduced. Maehly's second method is mentioned for the exponential function on the imaginary axis for the first time in the present work and its strategy for interpolation nodes correction provides excellent results for this application. The relation between the BRASIL algorithm and Maehly's second method, which is also mentioned for the first time in the present work to the authors' knowledge, can also have relevance for best approximations to other functions. 
 
Rational interpolants in~\ref{item:algbrib} and the AAA and AAA-Lawson methods~\ref{item:alglawson} are typically utilized in barycentric rational form. The poles, residues and zeros of barycentric rational functions are available in practice, cf.~\cite{NST18,Ho21}. Thus, computed approximants can be evaluated using product form, barycentric rational form, or in certain settings also partial fraction form (for an overview on partial fractions see~\cite[Chapter~7]{He74a}). In the context of approximation to the action of a related matrix exponential operator this will be the topic of a future work.

In the present work we show that the interpolation-based algorithm succeeds to compute the unitary best approximant for very large degrees in practice. The availability of high-degree approximants, together with the potential use of partial fraction decomposition, provides further advantages of the unitary best approximation over the Pad\'e approximation.

\subsection{Outline of the present work}
We recall some properties of unitary best approximation in Section~\ref{sec:urb}, and introduce the error in uniformity for an interpolatory setting in subsections~\ref{subsec:alternating} and~\ref{subsec:errsandwich}. This setting is particularly relevant for algorithms of type~\ref{item:algbrib} and approximants which are sufficiently close to the unitary best approximant.
In Corollary~\ref{cor:errsandwich} we show that the uniform error of a rational interpolant provides bounds which sandwich the uniform error of the unitary best approximant in practical settings, and we show that the error in uniformity can be used to determine convergence to the unitary best approximant.
A~priori estimates to determine $\omega$ for a given degree $n$ s.t.\ the unitary best approximant attains a given error level are provided in Section~\ref{sec:setting}.

In Section~\ref{sec:algorithms} we provide an interpolation-based algorithm~\ref{item:algbrib}, and we recall the AAA-Lawson method~\ref{item:alglawson}.
In Section~\ref{sec:convergence} performance of estimates and algorithms is illustrated and discussed using numerical experiments.

\section{Unitary best approximation}\label{sec:urb}

Throughout the present section we consider the degree $n$ and frequency $\omega$ to be given. Under the condition $\omega\in(0,(n+1)\pi)$ the unitary best approximant is uniquely characterized by an equioscillating phase error, as summarized in the following.

\noindent{\bf Approximation and phase errors.}
For a rational approximant $r(\mathrm{i} x)\approx \mathrm{e}^{\mathrm{i} \omega x}$ we refer to $|r(\mathrm{i} x)-\mathrm{e}^{\mathrm{i} \omega x}|$ as {\em approximation error}. For the uniform approximation error we make use of the norm notation~\eqref{eq:defnorm}, i.e., 
\begin{equation*}
\|r - \exp(\omega \cdot)\| = \max_{x\in[-1,1]}|r(\mathrm{i} x)-\mathrm{e}^{\mathrm{i} \omega x}|.
\end{equation*}
For a unitary rational function $r\in\mathcal{U}_n$ we remark $|r(\mathrm{i} x)-\mathrm{e}^{\mathrm{i} \omega x}|\leq 2$, and we also refer to $r$ as having a maximal approximation error if $\|r - \exp(\omega \cdot)\|=2$.

Following \cite[Proposition~4.2]{JS23u}, for $r\in\mathcal{U}_n$ with
$\| r - \exp(\omega \cdot) \| <2$,
there exists a unique phase function $g:\mathbb{R}\to\mathbb{R}$ with
\begin{equation}\label{eq:phaseerrdef}
r(\mathrm{i} x) = \mathrm{e}^{\mathrm{i} g(x)},~~~\text{s.t.}~~\max_{x\in[-1,1]} |g(x) - \omega x|<\pi,
\end{equation}
where $g(x) - \omega x$ is also referred to as {\em phase error} of $r(\mathrm{i} x) = \mathrm{e}^{\mathrm{i} g(x)}\approx \mathrm{e}^{\mathrm{i} \omega x}$.

\begin{remark}
Assume $\|r-\exp(\omega\cdot)\|<2$.
Since $r(\mathrm{i} x)/ \mathrm{e}^{\mathrm{i} \omega x} = \mathrm{e}^{\mathrm{i} (g(x)- \omega x)}$ and~\eqref{eq:phaseerrdef}, i.e., $g(x)- \omega x\in(-\pi,\pi)$ for $x\in[-1,1]$, the phase error satisfies
\begin{equation}\label{eq:phaseerrangle}
g(x) - \omega x = \operatorname{angle}(r(\mathrm{i} x)/ \mathrm{e}^{\mathrm{i} \omega x}) \in (-\pi,\pi),~~~x\in[-1,1].
\end{equation}
For the case $\|r-\exp(\omega\cdot)\|=2$, the phase function $g$ is only defined up to a multiple of $2\pi$ which carries over to~\eqref{eq:phaseerrangle}.
\end{remark}

\noindent{\bf Unitary best approximation.}
In the sequel we refer to the unitary best approximation as $\widetilde{r}$~\eqref{eq:uniformerror}.
For $\omega\in(0,(n+1)\pi)$ the unitary best approximant attains $\|\widetilde{r}-\exp(\omega\cdot)\|<2$ \cite[Proposition~4.5]{JS23u} and we let $\widetilde{g}$ denote its phase function, i.e.,
\begin{equation*}
\widetilde{r}(\mathrm{i} x) = \mathrm{e}^{\mathrm{i} \widetilde{g}(x)} \approx \mathrm{e}^{\mathrm{i} \omega x}.
\end{equation*}
Following \cite[Theorem~5.1]{JS23u}, for $\omega\in(0,(n+1)\pi)$ the unitary best approximation is uniquely characterized by an equioscillating phase error, i.e., there exist points $\eta_1<\ldots<\eta_{2n+2}\in[-1,1]$ with $\eta_1=-1$, $\eta_{2n+2}=1$ s.t.\ \cite[Theorem~5.1 and Proposition~7.5]{JS23u}
\begin{subequations}
\begin{equation}\label{eq:phaseerrequioscillates}
\widetilde{g}(\eta_j) - \omega \eta_j
= (-1)^{j+1} \max_{x\in[-1,1]} |\widetilde{g}(x) - \omega x|,\quad j=1,\ldots,2n+2,
\end{equation}
Moreover, following \cite[Corollary~5.2]{JS23u} the points $\eta_1,\ldots,\eta_{2n+2}$ are exactly the points in $[-1,1]$ at which the unitary best approximant attains its uniform approximation error, i.e., 
\begin{equation}\label{eq:urbequerr}
|\widetilde{r}(\mathrm{i} \eta_j) -  \mathrm{e}^{\mathrm{i} \omega \eta_j}|
= \|\widetilde{r} -  \exp(\omega \cdot)\|,~~~j=1,\ldots,2n+2,
\end{equation}
and there exist $2n+1$ interpolation nodes $x_1,\ldots,x_{2n+1}$ with
\begin{equation}\label{eq:urbinterpolate}
\widetilde{r}(\mathrm{i} x_j) = \mathrm{e}^{\mathrm{i} \omega x_j}, ~~~ x_j\in(\eta_j,\eta_{j+1}),~~~j=1,\ldots,2n+1.
\end{equation}
\end{subequations}

We make use of the following proposition for convergence results further below.
\begin{proposition}\label{prop:convergence}
Let $n$ denote a given degree and let $\omega\in(0,(n+1)\pi)$ be fixed.
Let $\{r_j\}_{j\in\mathbb{N}}$ denote a sequence of unitary $(n,n)$-rational approximants for which the uniform error converges to the uniform error of the unitary best approximant $\widetilde{r}$, i.e.,
\begin{equation}\label{eq:rjconvcond}
\|r_j-\exp(\omega\cdot)\| \to \|\widetilde{r}-\exp(\omega\cdot)\|,~~~\text{for}~~j\to \infty.
\end{equation}
Then, $r_j$ converges to the unitary best approximant, i.e.,
\begin{equation}\label{eq:rjconvresult}
\|r_j - \widetilde{r}\|\to 0,~~~\text{for}~~j\to \infty.
\end{equation}
\end{proposition}
\begin{proof}
Assuming~\eqref{eq:rjconvresult} does not hold true, then there exists a subsequence $\{r_{j_k}\}_{k\in\mathbb{N}}$ and $\alpha>0$ with
\begin{equation}\label{eq:rjconvcontalpha}
\|r_{j_k} - \widetilde{r}\|>\alpha,~~~\text{for}~~k\to\infty,
\end{equation}
We recall some arguments of~\cite[Proposition~3.1]{JS23u} to show convergence of a subsequence of $\{r_{j_k}\}_{k\in\mathbb{N}}$ to some $\widehat{r}\in\mathcal{U}_n$ in the following. The rational functions $r_{j_k}$ satisfy $r_{j_k}=p_{j_k}/q_{j_k}$ where $p_{j_k}$ and $q_{j_k}$ denote polynomials of degree $\leq n$ for $k\in\mathbb{N}$. Since $r_{j_k}$ is unitary we have $|p_{j_k}(\mathrm{i} x)|=|q_{j_k}(\mathrm{i} x)|$ for $x\in[-1,1]$ and we may assume that $p_{j_k}$ and $q_{j_k}$ are normalized, i.e., $\|p_{j_k}\|=\|q_{j_k}\|=1$. Since $p_{j_k}$ and $q_{j_k}$ are bounded in the set of polynomials, there exist convergent sub-sequences $p_{j_{k_\ell}}\to \widehat{p}$ and $q_{j_{k_\ell}}\to \widehat{q}$ for $\ell\to \infty$ for some polynomials $\widehat{p}$ and $\widehat{q}$ of degree $\leq n$. Define the $(n,n)$-rational function $\widehat{r}$ as $\widehat{r} = \widehat{p}/\widehat{q}$. Properties of $p_{j_{k_\ell}}$ and $q_{j_{k_\ell}}$ imply that $\widehat{r}$ is unitary, i.e., $\widehat{r}\in\mathcal{U}_n$. For $x\in[-1,1]$ with $\widehat{q}(\mathrm{i}x)\neq 0$ we get
\begin{equation}\label{eq:rjconvcont2}
\left|r_{j_{k_\ell}}(\mathrm{i} x) - \widehat{r}(\mathrm{i} x)\right|
=\left|p_{j_{k_\ell}}(\mathrm{i} x)/q_{j_{k_\ell}}(\mathrm{i} x)  - \widehat{p}(\mathrm{i} x)/\widehat{q}(\mathrm{i} x)\right| \to 0,~~~\text{for}~~\ell\to\infty.
\end{equation}
Moreover, for points $x\in[-1,1]$ for which this limit holds true we note
\begin{subequations}\label{eq:rjconvcont3}
\begin{equation}
|\widehat{r}(\mathrm{i} x)-\mathrm{e}^{\mathrm{i} \omega x}|
= \lim_{\ell\to\infty} \left|r_{j_{k_\ell}}(\mathrm{i} x)-\mathrm{e}^{\mathrm{i} \omega x}\right|.
\end{equation}
Taking the maximum over $x\in[-1,1]$ for the absolute value therein and making use of~\eqref{eq:rjconvcond} we observe
\begin{equation}
\lim_{\ell\to\infty} \left|r_{j_{k_\ell}}(\mathrm{i} x)-\mathrm{e}^{\mathrm{i} \omega x}\right|  \leq 
\lim_{\ell\to\infty} \|r_{j_{k_\ell}} - \exp(\omega \cdot)\|
= \|\widetilde{r}-\exp(\omega\cdot)\|.
\end{equation}
\end{subequations}
Due to continuity arguments the inequalities in~\eqref{eq:rjconvcont3} imply
\begin{equation*}
\|\widehat{r}-\exp(\omega\cdot)\| \leq \|\widetilde{r}-\exp(\omega\cdot)\|,
\end{equation*}
and since $\widetilde{r}$ is the unique rational function in $\mathcal{U}_n$ which minimizes this error we arrive at the identity $\widehat{r}\equiv\widetilde{r}$. The unitary best approximant is non-degenerate~\cite[Theorem~5.1]{JS23u} which particularly implies that $r_{j_{k_\ell}}$ is non-degenerate as well for sufficiently large $\ell$. Consequently, the denominator $q_{j_{k_\ell}}$ has no zeros on the imaginary axis for sufficiently large $\ell$ due to unitarity properties~\cite[Proposition~2.2]{JS23u}. Thus, the convergence~\eqref{eq:rjconvcont2} holds true uniformly for $x\in[-1,1]$, i.e.,
\begin{equation*}
\|r_{j_{k_\ell}} - \widetilde{r}\| \to 0.
\end{equation*}
This is in contradiction to~\eqref{eq:rjconvcontalpha} which proves our claim.
\end{proof}

\subsection{Alternating points}\label{subsec:alternating}
Computing a rational approximant that exactly satisfies the equioscillatory property~\eqref{eq:phaseerrequioscillates} is not practical in computer arithmetic. In the present subsection we provide some results for the case of a non-uniform alternating phase error.

The following auxiliary result is related to \cite[Proposition~4.2]{JS23u}.
\begin{proposition}\label{prop:aprxandphaseerr}
Let $r_1,r_2\in\mathcal{U}_n$ with $r_1(\mathrm{i} x)=\mathrm{e}^{\mathrm{i} g_1(x)}$ and $r_2(\mathrm{i} x)=\mathrm{e}^{\mathrm{i} g_2(x)}$ and assume $\|r_j-\exp(\omega\cdot)\|<2$ for $j\in\{1,2\}$.
Let $x_1,x_2\in[-1,1]$ be given points, then
\begin{subequations}
\begin{equation}\label{eq:r1x1r2x2le}
|r_1(\mathrm{i} x_1)-\mathrm{e}^{\mathrm{i} \omega x_1}| < |r_2(\mathrm{i} x_2)-\mathrm{e}^{\mathrm{i} \omega x_2}|,
\end{equation}
if and only if
\begin{equation}\label{eq:g1x1g2x2le}
|g_1(x_1)-\omega x_1| < |g_2(x_2)-\omega x_2|.
\end{equation}
\end{subequations}
In particular, the approximation error of a unitary rational function attains local maxima at points of local extrema of its phase error.
\end{proposition}
\begin{proof}
For $r_j(\mathrm{i} x)=\mathrm{e}^{\mathrm{i} g_j(x)}$ with $j\in\{1,2\}$ we note
\begin{subequations}\label{eq:Prop42xinproof}
\begin{equation}
|\mathrm{e}^{\mathrm{i} g_j(x)} - \mathrm{e}^{\mathrm{i} \omega x}|  =  |\mathrm{e}^{\mathrm{i} (g_j(x) - \omega x)/2} - \mathrm{e}^{-\mathrm{i} (g_j(x) - \omega x)/2} |
= 2  | \sin((g_j(x) - \omega x)/2 ) |.
\end{equation}
Following \cite[Proposition~4.2]{JS23u} the assumption $\|r_j-\exp(\omega\cdot)\|<2$ implies $|g_j(x) - \omega x|< \pi$ for $x\in[-1,1]$, and thus  
\begin{equation}
| \sin((g_j(x) - \omega x)/2 ) | = \sin(|g_j(x) - \omega x|/2 ).
\end{equation}
\end{subequations}
In particular, for the points $x_1,x_2\in[-1,1]$ the identities in~\eqref{eq:Prop42xinproof} imply
\begin{equation*}
|r_j(\mathrm{i} x_j) - \mathrm{e}^{\mathrm{i} \omega x_j}|
= 2  \sin(|g_j(x_j) - \omega x_j|/2 ),~~~j\in\{1,2\},
\end{equation*}
Making use of this identity, we observe that~\eqref{eq:r1x1r2x2le} is equivalent to
\begin{equation}\label{eq:r1x1r2x2le2}
\sin(|g_1(x_1) - \omega x_1|/2 ) < \sin(|g_2(x_2) - \omega x_2|/2 ).
\end{equation}
Since $|g_j(x_j) - \omega x_j|< \pi$ and the sine function is strictly monotonically increasing on $[0,\pi/2)$, we conclude that the inequality~\eqref{eq:g1x1g2x2le} holds true if and only if~\eqref{eq:r1x1r2x2le2}, respectively~\eqref{eq:r1x1r2x2le}, holds true which completes the proof.
\end{proof}

We refer to the points $\tau_1<\ldots<\tau_{n+2}$ as {\em alternating points} of the phase error if its sign is alternating at these points, i.e.,
\begin{equation}\label{eq:phaseerrnonunfaltern}
g(\tau_j) - \omega \tau_j
= (-1)^{j+\iota} |g(\tau_j) - \omega \tau_j|,\quad j=1,\ldots,2n+2,\quad \iota\in\{0,1\}.
\end{equation}

The following proposition is closely related to \cite[Proposition~5.4]{JS23u}.
\begin{proposition}
\label{prop:ValleePoussin}
Let $n$ denote a given degree and let $\omega\in(0,(n+1)\pi)$ be fixed.
Assume that $r\in\mathcal{U}_n$ has a non-maximal approximation error s.t.\ the phase function with $r(\mathrm{i} x)=\mathrm{e}^{\mathrm{i} g(x)}\approx \mathrm{e}^{\mathrm{i} \omega x}$ is well defined, and assume that the phase error of $r$ has $2n+2$ alternating points $\tau_1<\ldots<\tau_{2n+2}\in[-1,1]$ as in~\eqref{eq:phaseerrnonunfaltern}. Then
\begin{equation}\label{eq:phaseerrVP}
\min_{j=1,\ldots,2n+2} |g(\tau_j) - \omega \tau_j|
\leq  \max_{x\in[-1,1]} |\widetilde{g}(x) - \omega x|.
\end{equation}
Moreover, the approximation error satisfies
\begin{equation}\label{eq:approxerrVP}
\min_{j=1,\ldots,2n+2} |r(\mathrm{i} \tau_j) - \mathrm{e}^{\mathrm{i} \omega \tau_j}|
\leq   \|\widetilde{r} - \exp(\omega \cdot)\|.
\end{equation}
\end{proposition}
\begin{proof}
We first prove~\eqref{eq:phaseerrVP} by contradiction. Assuming the opposite of~\eqref{eq:phaseerrVP} we note
\begin{equation*}
|g(\tau_j) - \omega \tau_j|
> \max_{x\in[-1,1]} |\widetilde{g}(x) - \omega x|,~~~j=1,\ldots,2n+2.
\end{equation*}
Since the maximum therein is larger than the error of $\widetilde{g}$ at $\tau_1,\ldots,\tau_{2n+2}$, this implies 
\begin{equation}\label{eq:phaseerrVPcont}
|g(\tau_j) - \omega \tau_j|
> |\widetilde{g}(\tau_j) - \omega \tau_j|,~~~j=1,\ldots,2n+2.
\end{equation}
To simplify the proof we assume $\iota=0$ in the alternation property~\eqref{eq:phaseerrnonunfaltern} of the phase error of $r$. Combining~\eqref{eq:phaseerrnonunfaltern} with~\eqref{eq:phaseerrVPcont} we observe
\begin{align*}
&g(\tau_j) - \omega \tau_j < \widetilde{g}(\tau_j) - \omega \tau_j,~~~j=1,3,\ldots,~~\text{and}\\
&g(\tau_j) - \omega \tau_j > \widetilde{g}(\tau_j) - \omega \tau_j,~~~j=2,4,\ldots.
\end{align*}
Thus, there exists points $t_1,\ldots,t_{2n+1}$ with $t_j\in(\tau_j,\tau_{j+1})$ and
\begin{equation*}
g(t_j) = \widetilde{g}(t_j),~~~j=1,\ldots,2n+1,
\end{equation*}
and \cite[Proposition~5.3]{JS23u} implies $r=\widetilde{r}$, and consequently $g=\widetilde{g}$. 
Similar arguments hold true for the case $\iota=1$.
Since $g=\widetilde{g}$ is contradictory to~\eqref{eq:phaseerrVPcont}, this proves~\eqref{eq:phaseerrVP}.

We proceed to prove~\eqref{eq:approxerrVP}. Let $k\in\{1,\ldots,2n+2\}$ denote the index with
\begin{subequations}\label{eq:inproofVPids}
\begin{equation}
\min_{j=1,\ldots,2n+2} |g(\tau_j) - \omega \tau_j| = |g(\tau_k) - \omega \tau_k|
\end{equation}
Proposition~\ref{prop:aprxandphaseerr} with $r_1=r_2=r$ implies that the minimum of the approximation error at $\tau_1,\ldots,\tau_{2n+2}$ is attained at the same point as the minimum over the phase error. Thus,
\begin{equation}
\min_{j=1,\ldots,2n+2} |r(\mathrm{i} \tau_j) - \mathrm{e}^{\mathrm{i} \omega \tau_j}| = |r(\mathrm{i} \tau_k) - \mathrm{e}^{\mathrm{i} \omega \tau_k}|.
\end{equation}
In a similar manner, this proposition implies that $\widetilde{r}$ attains its maximal approximation and phase errors at the same point $\widehat{x}\in[-1,1]$, i.e.,
\begin{equation}
\max_{x\in[-1,1]} |\widetilde{g}(x) - \omega x| = |\widetilde{g}(\widehat{x}) - \omega \widehat{x}|,~~~\text{and}~~
 \|\widetilde{r} - \exp(\omega \cdot)\| = |\widetilde{r}(\mathrm{i} \widehat{x}) - \mathrm{e}^{\mathrm{i} \omega \widehat{x}}|.
\end{equation}
\end{subequations}
Thus,~\eqref{eq:phaseerrVP} corresponds to
\begin{equation}\label{eq:phaseerrVP2}
|g(\tau_k) - \omega \tau_k|
\leq |\widetilde{g}(\widehat{x}) - \omega \widehat{x}|,
\end{equation}
and~\eqref{eq:approxerrVP} corresponds to
\begin{equation}\label{eq:approxerrVP2}
|r(\mathrm{i} \tau_k) - \mathrm{e}^{\mathrm{i} \omega \tau_k}|
\leq |\widetilde{r}(\mathrm{i} \widehat{x}) - \mathrm{e}^{\mathrm{i} \omega \widehat{x}}|.
\end{equation}
Applying Proposition~\ref{prop:aprxandphaseerr} for $r$, $g$, $\tau_k$ and $\widetilde{r}$, $\widetilde{g}$, $\widehat{x}$ shows that~\eqref{eq:phaseerrVP2} holds true if and only if~\eqref{eq:approxerrVP2} holds true. This equality carries over to~\eqref{eq:phaseerrVP} and~\eqref{eq:approxerrVP} due to the identities in~\eqref{eq:inproofVPids}  which completes the proof.
\end{proof}

\subsection{Rational interpolants and the error in uniformity}\label{subsec:errsandwich}

In the present subsection we consider the setting of rational interpolants as used for algorithms of type~\ref{item:algbrib} (discussed in more detail in Subsection~\ref{subsec:brib} below). In particular, the notation $x_1,\ldots,x_{2n+1}$ and $\eta_1,\ldots,\eta_{2n+2}$ for the interpolation nodes~\eqref{eq:urbinterpolate} and equioscillation points~\eqref{eq:phaseerrequioscillates}, respectively, of the unitary best approximant $\widetilde{r}$ is re-used for rational interpolants in a more general setting as following.

We consider distinct nodes $x_1,\ldots,x_{2n+1}\in(-1,1)$ in ascending order and a $(n,n)$-rational function which interpolates $\mathrm{e}^{\mathrm{i} \omega x}$ in these nodes, i.e.,
\begin{subequations}\label{eq:nonuniforminterpsetting}
\begin{equation}\label{eq:nnsinterpolate}
r(\mathrm{i} x_j) = \mathrm{e}^{\mathrm{i} \omega x_j},~~~j=1,\ldots,2n+1.
\end{equation}
Moreover, let $\eta_1,\ldots,\eta_{2n+2}$ denote points of intermediate maxima of the approximation error $|r(\mathrm{i} x)-\mathrm{e}^{\mathrm{i} \omega x}|$, s.t.
\begin{equation}
-1\leq \eta_1<x_1<\eta_2<\ldots<x_{2n+1}<\eta_{2n+2}\leq 1,
\end{equation}
with
\begin{equation}\label{eq:localmaxerr}
\begin{aligned}
\max_{x\in[-1,x_1)} |r(\mathrm{i} x)-\mathrm{e}^{\mathrm{i} \omega x}| &= |r(\mathrm{i} \eta_1)-\mathrm{e}^{\mathrm{i} \eta_1}|,\\
\max_{x\in(x_{j-1},x_j)} |r(\mathrm{i} x)-\mathrm{e}^{\mathrm{i} \omega x}| &= |r(\mathrm{i} \eta_j)-\mathrm{e}^{\mathrm{i} \eta_j}|,~~\text{for $j=2,\ldots,2n+1$, and}\\
\max_{x\in(x_{2n+1},1]} |r(\mathrm{i} x)-\mathrm{e}^{\mathrm{i} \omega x}| &=|r(\mathrm{i} \eta_{2n+2})-\mathrm{e}^{\mathrm{i} \eta_{2n+2}}|.
\end{aligned}
\end{equation}
\end{subequations}
We use the notation $\varepsilon_j$ for the maxima attained at $\eta_j$ in the following, i.e.,
\begin{equation}\label{eq:epsj}
\varepsilon_j = |r(\mathrm{i} \eta_j)-\mathrm{e}^{\mathrm{i}\omega \eta_j}|,~~~j=1,\ldots,2n+2.
\end{equation}
Making use of this notation, we note that the uniform error of $r$ corresponds to
\begin{equation}\label{eq:errnormismaxetaj}
\|r-\exp(\omega \cdot)\| 
=\max_{k=1,\ldots,2n+2} \varepsilon_k.
\end{equation}

In contrast to polynomial interpolation, solutions to the rational interpolation problem~\eqref{eq:nnsinterpolate} might not exist, cf.~\cite[Section~2]{Be70}, \cite{MW60b, Gu90} and others. However, for the present work we assume existence of rational interpolants for the underlying nodes, which is certainly given if these nodes correspond to, or are sufficiently close to, the interpolation nodes of the unitary best approximant~\eqref{eq:urbinterpolate}.
In particular, if the interpolation nodes $x_1,\ldots,x_{2n+1}$ in~\eqref{eq:nnsinterpolate} correspond to the interpolation nodes of the unitary best approximant~\eqref{eq:urbinterpolate}, then $r\equiv \widetilde{r}$ since the rational interpolant is unique in the non-degenerate case which holds true for $\widetilde{r}$. Consequently, in this case the points $\eta_1,\ldots,\eta_{2n+2}$ in~\eqref{eq:localmaxerr} correspond to the equioscillation points of the unitary best approximant~\eqref{eq:phaseerrequioscillates}.

Following~\cite[Proposition~2.1]{JS24} (or~\cite[Proposition~2.4]{JS23u} covering a more general setting) 
the rational interpolant $r$ satisfying~\eqref{eq:nnsinterpolate} is unitary, i.e., $r\in\mathcal{U}_n$. In line with~\cite[Section~6]{JS23u} we refer to a unitary rational function $r$ as symmetric if
\begin{equation}\label{eq:rsym}
r(-z)^{-1} = r(z),~~~z\in\mathbb{C}.
\end{equation}
In the following auxiliary result, we show symmetry of rational interpolants in a certain setting.
\begin{proposition}\label{prop:symmetry}
Let the nodes $x_1,\ldots,x_{2n+1}$ be mirrored around zero, i.e.,
\begin{equation}\label{eq:xjmirrored}
x_{n+1}=0,~~~\text{and}~~x_j = -x_{2n+2-j},~~~j=1,\ldots,n,
\end{equation}
and let $r$ be a non-degenerate $(n,n)$-rational function which satisfies the interpolation problem~\eqref{eq:nnsinterpolate}. Then, $r$ is symmetric~\eqref{eq:rsym}.
\end{proposition}
\begin{proof}
We define $\zeta(z) = r(-z)^{-1}$, and note that $\zeta$ corresponds to a $(n,n)$-rational function. We proceed to show that $\zeta$ solves the interpolation problem~\eqref{eq:nnsinterpolate}, i.e.,
\begin{equation}\label{eq:nnsinterpolate2}
\zeta(\mathrm{i} x_j) = \mathrm{e}^{\mathrm{i} \omega x_j},~~~j=1,\ldots,2n+1.
\end{equation}
Since the nodes $x_j$ are mirrored around zero, i.e., $x_j=-x_{2n+2-j}$ for $j=1,\ldots,n$, the first $n$ conditions in~\eqref{eq:nnsinterpolate} are equivalent to
\begin{subequations}\label{eq:nnsinterpolate2x}
\begin{equation}\label{eq:nnsinterpolate20}
\zeta(-\mathrm{i} x_{2n+2-j}) = \mathrm{e}^{-\mathrm{i} \omega x_{2n+2-j}},~~~j=1,\ldots,n.
\end{equation}
Making use if the definition of $\zeta$, and the identity $\mathrm{e}^{-z}=(\mathrm{e}^{z})^{-1}$, the identity~\eqref{eq:nnsinterpolate20} is equivalent to
\begin{equation}
r(\mathrm{i} x_{2n+2-j})^{-1} = (\mathrm{e}^{\mathrm{i} \omega x_{2n+2-j}})^{-1},
\end{equation}
\end{subequations}
which holds true since $r$ satisfies the interpolation condition~\eqref{eq:nnsinterpolate}. The identities in~\eqref{eq:nnsinterpolate2x} show that~\eqref{eq:nnsinterpolate2} holds true for $j=1,\ldots,n$, and similar arguments hold true for $j=n+2,\ldots,2n+1$. For the node $x_{n+1}=0$ the interpolation condition reads $\zeta(0) = 1$, which holds true since $\zeta(0) = r(0)^{-1}$ by definition, and $r(0)=1$ due to~\eqref{eq:nnsinterpolate}.

Thus, the identities~\eqref{eq:nnsinterpolate2} hold true and $\zeta$ satisfies the same interpolation problem as $r$. However, $r$ uniquely solves this interpolation problem since we assume that $r$ is non-degenerate, cf.~\cite{Gu90}. This shows $\zeta\equiv r$, and thus, $r$ is symmetric~\eqref{eq:rsym} which proves our assertion.
\end{proof}
Since rational interpolants to $\mathrm{e}^{\mathrm{i}\omega x}$ are unitary, the symmetry property~\eqref{eq:rsym} is equivalent to
\begin{equation*}
\overline{r(-\mathrm{i} x)} = r(\mathrm{i} x),~~~x\in\mathbb{R}.
\end{equation*}
Thus, the approximation error of $r(\mathrm{i} x)\approx\mathrm{e}^{\mathrm{i} \omega x}$ is an  is an even function in $x$, i.e.,
\begin{equation}\label{eq:errormirrored}
|r(\mathrm{i} x)-\mathrm{e}^{\mathrm{i} \omega x}| = |r(-\mathrm{i} x)-\mathrm{e}^{-\mathrm{i} \omega x}|,~~~x\in\mathbb{R}.
\end{equation}
Consequently, in a symmetric setting the intermediate points of maximal error $\eta_1,\ldots,\eta_{2n+2}$~\eqref{eq:localmaxerr} are mirrored around zero, i.e., we may choose the points $\eta_j$ s.t.
\begin{subequations}\label{eq:etajmirroredx}
\begin{equation}\label{eq:etajmirrored}
\eta_j = -\eta_{2n+3-j},~~~ j=1,\ldots,n+1,
\end{equation}
Due to~\eqref{eq:errormirrored} the respective errors satisfy
\begin{equation}\label{eq:epsjmirrored}
\varepsilon_j
= \varepsilon_{2n+3-j},~~~ j=1,\ldots,n+1.
\end{equation}
\end{subequations}

{\bf Error in uniformity.}
Considering the setting of~\eqref{eq:nonuniforminterpsetting} we define
the {\em error in uniformity}
$\delta$ as
\begin{equation}\label{eq:derapproxerr}
\delta = 1 - \min_{j=1,\ldots,2n+2} \varepsilon_j \big/ \max_{k=1,\ldots,2n+2} \varepsilon_k.
\end{equation}

In the following corollary we make use of $\delta$ to enclose the error of the unitary best approximant by the error of a rational interpolant. We recall that a rational interpolant as in~\eqref{eq:nonuniforminterpsetting} is unitary, and under the assumption that $\|r-\exp(\omega\cdot)\|< 2$ it satisfies the representation $r(\mathrm{i} x)=\mathrm{e}^{\mathrm{i} g(x)}$ for a phase function~$g$.
\begin{corollary}[to propositions~\ref{prop:convergence} and~\ref{prop:ValleePoussin}] \label{cor:errsandwich}
Let $r$ satisfy~\eqref{eq:nonuniforminterpsetting}. Assume $r$ has a non-maximal error, i.e., $\|r-\exp(\omega\cdot)\|< 2$ and assume the phase error of $r$ has an alternating sign at the points $\eta_1,\ldots,\eta_{2n+2}$ in line with~\eqref{eq:phaseerrnonunfaltern}, i.e.,
\begin{equation*}
g(\eta_j) - \omega \eta_j
= (-1)^{j+\iota} |g(\eta_j) - \omega\eta_j|,\quad j=1,\ldots,2n+2,\quad \iota\in\{0,1\}.
\end{equation*}
Let the error in uniformity $\delta$ be defined as in~\eqref{eq:derapproxerr}. Then, the error of the unitary best approximant is sandwiched by
\begin{equation}\label{eq:dererrsandwich}
(1-\delta) \|r-\exp(\omega \cdot)\|  \leq  \|\widetilde{r} - \exp(\omega \cdot)\| \leq \|r-\exp(\omega \cdot)\| .
\end{equation}
In particular, the relative distance between the errors of the computed approximant and the unitary best approximant is bounded by
\begin{equation}\label{eq:errdistbound}
\frac{|\|r-\exp(\omega \cdot)\| - \|\widetilde{r}-\exp(\omega \cdot)\||}{\|r-\exp(\omega \cdot)\|}
\leq \delta .
\end{equation}
Moreover, let $\{r_j\}_{j\in\mathbb{N}}$ denote a sequence of rational interpolants s.t.\ $r_j$ satisfies the conditions above. Let $\delta_j$ refer to the error in uniformity of $r_j$ and assume $\delta_j\to 0$ for $j\to \infty$. Then, $r_j$ converges to the unitary best approximant $\widetilde{r}$, i.e.,
\begin{equation}\label{eq:rjwithdeltato0converges}
\|r_j-\widetilde{r}\|\to 0.
\end{equation}
\end{corollary}
\begin{proof}
Substituting the identity~\eqref{eq:errnormismaxetaj} in the definition~\eqref{eq:derapproxerr} we simplify the left-hand side in~\eqref{eq:dererrsandwich} to
\begin{equation}\label{eq:dertoerr}
(1-\delta) \|r-\exp(\omega \cdot)\|  = \min_{j=1,\ldots,2n+2} |r(\mathrm{i} \eta_j)-\mathrm{e}^{\mathrm{i} \eta_j}|
\end{equation}
Under the assumption that the corresponding phase error changes its sign at local extrema, Proposition~\ref{prop:ValleePoussin} implies
\begin{equation*}
\min_{j=1,\ldots,2n+2} \varepsilon_j \leq  \|\widetilde{r} - \exp(\omega \cdot)\|,
\end{equation*}
and together with~\eqref{eq:dertoerr} this proves the lower bound in~\eqref{eq:dererrsandwich}. Moreover, the upper bound in~\eqref{eq:dererrsandwich} holds true holds true since the rational interpolant $r$ is unitary and $\widetilde{r}$ denotes the unitary best approximation. Moreover,~\eqref{eq:errdistbound} directly follows from~\eqref{eq:dererrsandwich}.

To show the claim~\eqref{eq:rjwithdeltato0converges}, we note that the upper bound~\eqref{eq:errdistbound}, unitarity $\|r_j-\exp(\omega \cdot)\|\leq 2$, and the assumption $\delta_j\to 0$ implies
\begin{equation}
|\|r_j-\exp(\omega \cdot)\| - \|\widetilde{r}-\exp(\omega \cdot)\||
\leq 2 \delta_j \to 0. 
\end{equation}
Consequently, Proposition~\ref{prop:convergence} shows~\eqref{eq:rjwithdeltato0converges} which completes the proof.
\end{proof}

The results of Corollary~\ref{cor:errsandwich} have some relevance when computing the unitary best approximant by rational interpolation in successively corrected nodes, approach~\ref{item:algbrib}. Namely, to detect whether the computed approximant has a uniform error similar to the uniform error of the unitary best approximation and whether it converges to the unitary best approximant. For more details we refer to Subsection~\ref{subsec:brib}.
The interpolatory setting for defining the error in uniformity does not apply for approximants computed by the AAA-Lawson method, approach~\ref{item:alglawson}, in general. However, interpolation properties carry over from the unitary best approximant in case the computed approximant is sufficiently close to the unitary best approximant. Thus, the results of Corollary~\ref{cor:errsandwich} suit well to detect and evaluate convergence for approximants computed by the interpolation-based algorithm as well as the AAA-Lawson method.

\section{A priori estimates for \texorpdfstring{$\omega$}{w}}\label{sec:setting}

Success of presented algorithms to compute the unitary best approximant in computer arithmetic relies on a proper choice of the frequency $\omega$ w.r.t.\ the degree~$n$. The following two cases usually yield difficulties in this context:
\begin{enumerate}[label={(P\arabic*)}]
\item\label{item:difficultieswlarge} Frequencies $\omega$ close to $(n+1)\pi$. For $\omega\to (n+1)\pi$ the unitary best approximant approaches a degenerate case, i.e., $\widetilde{r}\equiv 1$, which potentially yields difficulties for computation.
\item\label{item:difficultieswsmall} The case $\omega\to 0$. Since the uniform approximation error of the unitary best approximant converges to zero with asymptotic order $\mathcal{O}(\omega^{2n+1})$ for $\omega \to 0$, the unitary best approximant quickly attains a uniform error below computer precision for $\omega<(n+1)\pi$.
This makes it particularly challenging to detect points of maximal errors for small $\omega$ for the interpolation-based algorithm or to uniquely solve least-squares problems which occur as subroutines of the AAA-Lawson method.
\end{enumerate}
The problems described in~\ref{item:difficultieswlarge} and~\ref{item:difficultieswsmall} can usually be avoided by choosing $\omega$ s.t.\ the unitary best approximant attains an error sufficiently below the maximal error of two and sufficiently above computer precision, respectively.
Thus, to successfully compute unitary best approximants it has some relevance to determine $\omega$ to attain a given error level, i.e.,
\begin{equation}\label{eq:problemfindw}
\text{for given $n$ and $\varepsilon>0$, find $\omega>0$ s.t.}~~\min_{r\in\mathcal{U}_n} \|r - \exp(\omega \cdot)\| = \varepsilon.
\end{equation}
The error of the unitary best approximant is monotonically increasing and continuous in $\omega$, approaches two for $\omega\to(n+1)\pi$ and vanishes for $\omega\to 0$, cf.~\cite[Section~3]{JS23u}. Thus, there exist solutions $\omega>0$ to~\eqref{eq:problemfindw} for $\varepsilon\leq 2$. 
Considering the focus of the present section, it is sufficient to determine $\omega$ s.t.
\begin{equation*}
\min_{r\in\mathcal{U}_n} \|r - \exp(\omega \cdot)\| \approx \varepsilon.
\end{equation*}
However, the presented estimates show to be accurate in most cases and also can be used for a~priori error estimation.

Computing unitary best approximants for $\omega$ as in~\eqref{eq:problemfindw} with an error $\varepsilon$ above computer precision and below the maximal error of two covers most practical cases considering applications of unitary best approximation. We remark that the unitary best approximant for given $\omega_1$ also provides an approximant to $\mathrm{e}^{\mathrm{i} \omega_2 x}$ attaining the same or higher accuracy for $\omega_2 < \omega_1$, or respectively, an approximation to $\mathrm{e}^{\mathrm{i} x}$ on $[-\omega_2,\omega_2]\subset [-\omega_1,\omega_1]$ when considering the setting of Remark~\ref{rmk:approxww}. Thus, in practice unitary best approximants computed for $\omega$ satisfying~\eqref{eq:problemfindw} also provide accurate approximants for smaller $\omega$, and restrictions to $\omega$ are not critical in this context.

In the following subsections we introduce two approaches to solve~\eqref{eq:problemfindw}, i.e., an approach based on experimental data and an approach based on the asymptotic error behavior of the unitary best approximant. Performance of these estimates is illustrated for numerical examples in \figurename~\ref{fig:testwest} in Section~\ref{sec:convergence} further below.

\subsection{Estimate based on experimental data}

In the present subsection we consider an estimate for $\omega$ in~\eqref{eq:problemfindw} based on experimental data. For given $n$ and $\varepsilon$ we let $\omega=\omega(n,\varepsilon)$ denote the solution to~\eqref{eq:problemfindw} and we define the scaling factor $\xi(n,\varepsilon) \in (0,1)$ as
\begin{equation*}
\xi(n,\varepsilon) = \omega(n,\varepsilon)/((n+1)\pi),
~~~\text{for $\omega(n,\varepsilon)$ satisfying~\eqref{eq:problemfindw}}.
\end{equation*}
Values of $\xi(n,\varepsilon)$ and $-\log\xi(n,\varepsilon)$ plotted in \figurename~\ref{fig:xistep1} imply that $-\log\xi(n,\varepsilon)$ satisfies a linear behavior as a function of $n$ in a logarithmic sense for various  $\varepsilon$, i.e., 
\begin{equation}\label{eq:xiansatz1}
\log(-\log\xi(n,\varepsilon)) \approx \log \widetilde{a}_\varepsilon + \widetilde{b}_\varepsilon \log n,
\end{equation}
for parameters $\widetilde{a}_\varepsilon$ and $\widetilde{b}_\varepsilon$ depending on $\varepsilon$. We proceed to consider polynomials in $\log \varepsilon$ for these parameters, namely,
\begin{equation}\label{eq:pabformulae0}
\widetilde{a}_\varepsilon = p_a(\log(\varepsilon)),~~~\text{and}~~
\widetilde{b}_\varepsilon = p_b(\log(\varepsilon)),
\end{equation}
where $p_a$ and $p_b$ are to be determined.
Substituting~\eqref{eq:pabformulae0} for $\widetilde{a}_\varepsilon$ and $\widetilde{b}_\varepsilon$ in~\eqref{eq:xiansatz1} and resolving for $\xi$, we arrive at
\begin{equation*}
\xi(n,\varepsilon) \approx \exp\left( -p_a(\log(\varepsilon)) n^{p_b(\log(\varepsilon))}\right).
\end{equation*}

\begin{figure}
    \centering
    \includegraphics[width=0.9\textwidth]{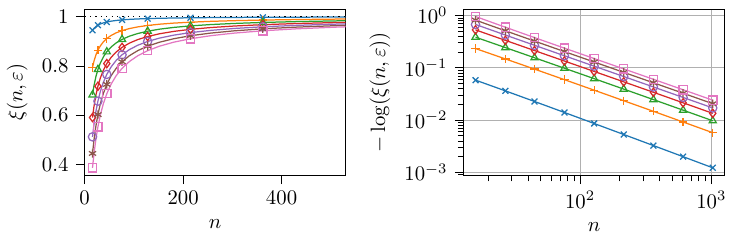}
    \caption{These plots show computed data for the scaling factor $\xi=\xi(n,\varepsilon)$ with $\omega = (n+1)\pi\xi$ s.t.\ $\omega$ satisfies the problem~\eqref{eq:problemfindw}. The lines in these plots illustrate $\xi(n,\varepsilon)$ (left) and $-\log (\xi(n,\varepsilon))$ (right) over $n$ for different values of $\varepsilon$. In particular, the lines marked by symbols '$+$, $\times$, $\Delta$, $\Diamond$, $\circ$, $*$' and '$\square$' correspond to $\varepsilon = 10^{0}, 10^{-2}, \ldots, 10^{-12}$, respectively.}
    \label{fig:xistep1}
\end{figure}

We proceed to derive an estimate $\omega_e$ for~\eqref{eq:problemfindw} by choosing $p_a$ and $p_b$ based on experimental data as following. We first compute $\omega(n,\varepsilon)$ for various values of $n$ and $\varepsilon$ s.t.\ the computed best approximants sandwich the error objective $\varepsilon$ as in~\eqref{eq:dererrsandwich} with $\delta<10^{-6}$. This procedure was done for $\varepsilon=10^{-14},\sqrt{10}\cdot10^{-14}, 10^{-13},\ldots,1$ with $n$ in a set of $25$ geometrically spaced values from $n=16$ to $n=1024$, using higher precision arithmetic for computations.
For each choice of $\varepsilon$ we compute $\widetilde{a}_\varepsilon$ and $\widetilde{b}_\varepsilon$ as in~\eqref{eq:xiansatz1} using a least-squares polynomial fitting method of degree $1$ for the computed sequence of $\omega(n,\varepsilon)$, $n=16,\ldots,1024$. We then compute $p_a$ and $p_b$ using a least-squares polynomial fitting method for the previously computed $\widetilde{a}_\varepsilon$ and $\widetilde{b}_\varepsilon$, $\varepsilon=10^{-14},\ldots,1$,  using polynomials of degrees $10$ and $11$, respectively. Using this approach to estimate the scaling factor $\xi$, and respectively $\omega=(n+1)\pi\xi$, we arrive at the estimate
\begin{subequations}\label{eq:est2w}
\begin{equation}
\omega_e(n,\varepsilon) = (n+1)\pi \exp\left( -p_a(\log(\varepsilon)) n^{p_b(\log(\varepsilon))}\right),
\end{equation}
with
\begin{equation}\label{eq:pabformulae}
\begin{aligned}
&p_a(t) = a_0 + a_1 t+ \ldots + b_{10} t^{10},~~~
p_b(t) = b_0 + b_1 t + \ldots + b_{11} t^{11},\\
&\text{and $a_j,b_j$ as in Table~\ref{table:coefpab}.}
\end{aligned}
\end{equation}
\end{subequations}
Moreover, for $\varepsilon<10^{-14}$ (which might has relevance when using higher precision arithmetic) we extrapolate the data computed for $10^{-14},\ldots,\sqrt{10}\cdot 10^{-13}$ which yields linear variants for $p_a$ and $p_b$ with coefficients shown at the bottom of Table~\ref{table:coefpab}. The polynomials $p_a$ and $p_b$ together with $\widetilde{a}_\varepsilon$ and $\widetilde{b}_\varepsilon$ are illustrated in \figurename~\ref{fig:xistep2}.

\begin{figure}
    \centering
    \includegraphics[width=0.9\textwidth]{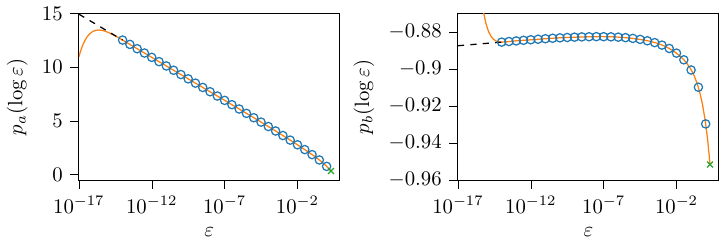}
    \caption{The solid lines in these plots show the polynomials $p_a$ (left) and $p_b$ (right) corresponding to~\eqref{eq:pabformulae}. The experimental values of $\widetilde{a}_\varepsilon$ (left) and $\widetilde{b}_\varepsilon$ (right) for $\varepsilon = 10^{-14}, \sqrt{10}\cdot 10^{-14}, 10^{-13}, \ldots, 1$ are marked by '$\circ$' symbols. The values of $p_a(\log \varepsilon)$ and $p_b(\log \varepsilon)$ for $\varepsilon=2$ are marked by '$\times$' symbols. The dashed line shows extrapolated data for $\varepsilon<10^{-14}$.}
    \label{fig:xistep2}
\end{figure}

\begin{table}
\caption{The tables at the top of this figure show the coefficients $a_j$ and $b_j$ of the polynomials $p_a$ and $p_b$, respectively, in~\eqref{eq:pabformulae}. The tables at the bottom show coefficients for $p_a$ and $p_b$ based on extrapolation suitable for the case $\varepsilon<10^{-14}$.}
\label{table:coefpab}
\begin{tabular}{|c|c|}
\hline
$j$&$a_j$\\\hline
0&$7.7325733748629055\cdot 10^{-1}$\\
1&$-5.777408873924058\cdot 10^{-1}$\\
2&$-6.860343132683391\cdot 10^{-2}$\\
3&$-1.4498935965331126\cdot 10^{-2}$\\
4&$-2.0017032381431967\cdot 10^{-3}$\\
5&$-1.792107115710027\cdot 10^{-4}$\\
6&$-1.0467338695044732\cdot 10^{-5}$\\
7&$-3.9545380249348945\cdot 10^{-7}$\\
8&$-9.304919862544986\cdot 10^{-9}$\\
9&$-1.2386694533170104\cdot 10^{-10}$\\
10&$-7.121569685837123\cdot 10^{-13}$\\
\hline\hline
$j$&$a_j$ extrapolation\\\hline
0&$1.2653161350741573\cdot 10^{0}$\\
1&$-3.4960298585304206\cdot 10^{-1}$\\
\hline
\end{tabular}
\begin{tabular}{|c|c|}
\hline
$j$&$b_j$\\\hline
0&$-9.296235152950844\cdot 10^{-1}$\\
1&$-2.4713673601660884\cdot 10^{-2}$\\
2&$-8.54706119111975\cdot 10^{-3}$\\
3&$-2.0382018252632794\cdot 10^{-3}$\\
4&$-3.2440829161667404\cdot 10^{-4}$\\
5&$-3.459972041530702\cdot 10^{-5}$\\
6&$-2.4972665972026706\cdot 10^{-6}$\\
7&$-1.2203258361585594\cdot 10^{-7}$\\
8&$-3.971747584379515\cdot 10^{-9}$\\
9&$-8.237224551239086\cdot 10^{-11}$\\
10&$-9.84139635152686\cdot 10^{-13}$\\
11&$-5.152327054589812\cdot 10^{-15}$\\
\hline\hline
$j$&$b_j$ extrapolation\\\hline
0&$-8.76285182160704\cdot 10^{-1}$\\
1&$2.8332004893961966\cdot 10^{-4}$\\
\hline
\end{tabular}
\centering
\end{table}

\subsection{Asymptotic estimate}
The asymptotic error of the unitary best approximant for $\omega\to0$ is studied in~\cite[Section~8]{JS23u} and these results are based on related asymptotic errors of Pad\'e approximation, cf.~\cite{Ja24cheb}. The leading order term of the asymptotic error is suggested to be used as an error estimate in~\cite[eq.~(8.10)]{JS23u}, i.e.
\begin{equation*}
\min_{r\in\mathcal{U}_n} \|r - \exp(\omega \cdot)\| \approx \frac{2(n!)^2 (\omega/2)^{2n+1}}{(2n)!(2n+1)!}.
\end{equation*}
We derive an estimate $\omega_a$ for~\eqref{eq:problemfindw} by setting the error estimate above equal to $\varepsilon$ and resolving this equation for $\omega$, i.e.,
\begin{subequations}\label{eq:asymestw}
\begin{equation}
\omega_a(n,\varepsilon) = 2\left(\frac{ \varepsilon (2n)!(2n+1)!}{2(n!)^2} \right)^{1/(2n+1)}.
\end{equation}
To avoid arithmetic overflow for larger $n$ in practice, this formula can be evaluated as
\begin{equation}
\omega_a(n,\varepsilon) = 2\exp\left(\frac{\log(\varepsilon(2n+1)/2) + 2 \sum_{j=1}^n \log (n+j)}{2n+1}\right).
\end{equation}
\end{subequations}

\section{Algorithms}\label{sec:algorithms}

We consider an interpolation-based algorithm~\ref{item:algbrib} and the AAA-Lawson method~\ref{item:alglawson} to compute the unitary best approximant in subsections~\ref{subsec:brib} and~\ref{subsec:alglawson}, respectively.

\subsection{Interpolation-based algorithm}\label{subsec:brib}
In the present subsection we consider the approach~\ref{item:algbrib} to compute the unitary best approximant. Namely, by rational interpolation in nodes which are successively corrected with the aim that the maximal error on the intermediate subintervals approaches the uniform error.
This approach goes back to~\cite{Ma63} for rational Chebyshev approximation to real functions and relies on existence of interpolation nodes and intermediate points of uniform maximal error which are a consequence of an equioscillatory property of rational Chebyshev approximation in real settings, cf.~\cite[Chapter~24]{Tre13}.
While in complex settings such results do not hold true in general, the unitary best approximant satisfies necessary properties to utilize the idea of Maehly's second method. 
In particular, interpolation~\eqref{eq:urbinterpolate} and equal maximal errors on intermediate subintervals~\eqref{eq:urbequerr}.
The unitary best approximant is characterized by an equioscillating phase error~\eqref{eq:phaseerrequioscillates}, and the identities~\eqref{eq:urbequerr} and~\eqref{eq:urbinterpolate} can be understood as consequences of this equioscillatory property. Nevertheless, various variants of Maehly's second method, aiming to satisfy~\eqref{eq:urbequerr} and~\eqref{eq:urbinterpolate}, succeed to provide approximants with an equioscillating phase error~\eqref{eq:phaseerrequioscillates} in practice. It seems to be crucial in this context that rational interpolants to $\mathrm{e}^{\mathrm{i} \omega x}$ in real points are unitary, cf.~\cite{JS23u,JS24}.

To further describe the interpolation-based algorithm in the present subsection we make use of the notation in~\eqref{eq:nonuniforminterpsetting} for the rational interpolant, the underlying interpolation nodes and the points of maximal error at intermediate subintervals. Moreover, we assume rational interpolants exist and are non-degenerate in practice, in particular, satisfying conditions of Proposition~\ref{prop:symmetry} if the underlying interpolation nodes are mirrored around zero. The main iteration of the interpolation-based algorithm is sketched in \algorithmname~\ref{alg:bribiteration}.

The procedure to correct interpolation nodes is repeated until the error in uniformity $\delta$ defined as in~\eqref{eq:derapproxerr} is smaller than a given tolerance or a maximal number of iterations is reached. In particular, the error in uniformity $\delta$ is available on the run for the interpolation-based algorithm and can be used as a stopping criterion in this context. Applying Corollary~\ref{cor:errsandwich} under the assumption that the respective conditions hold true, we note that a small error in uniformity entails that the uniform error of the computed approximant is close to the uniform error of the unitary best approximant.
Moreover, Corollary~\ref{cor:errsandwich} shows that the interpolant computed within the iteration of \algorithmname~\ref{alg:bribiteration} converges to the unitary best approximant in case $\delta\to0$.
A similar stopping criterion is also used in the BRASIL algorithm~\cite{Ho21} for best approximation to real functions with some lack of theory.

A crucial part considering the success of the interpolation-based algorithm is a suitable strategy to correct the interpolation nodes at each iteration, i.e., the procedure used in line~\ref{alg:line:nodecorrection} of \algorithmname~\ref{alg:bribiteration}. Besides Maehly's second method, Algorithm~G~\cite{Fra76} and the BRASIL algorithm~\cite{Ho21} also utilize the approach~\ref{item:algbrib} to compute rational Chebyshev approximants to real functions using different strategies for interpolation nodes correction. We consider the strategies of Maehly's second method and the BRASIL algorithm for interpolation nodes correction in the present work, namely in Subsection~\ref{subsec:brib:nodecorrection}. The performance of different strategies is illustrated in Section~\ref{sec:convergence}. In particular, the strategy of Maehly's second method shows to perform best if the intermediate points of maximal error are properly detected, the approximation error is non-maximal and the phase error has an alternating sign. On the other hand, the strategy of the BRASIL algorithm re-scales subintervals based on the intermediate maximal error and eventually succeeds interpolation nodes correction even if points of intermediate maximal errors are not properly detected in the initial phase. Thus, we suggest to combine these strategies as in Subsection~\ref{subsec:bestalgorithm} for best overall performance.

\begin{algorithm}
\caption{The main iteration of the interpolation-based algorithm to compute the unitary best approximant to $\mathrm{e}^{\mathrm{i} \omega x}$.}
\label{alg:bribiteration}
\begin{algorithmic}[1]
\Require{$n,w,x_1^{\text{init}},\ldots,x_{2n+1}^{\text{init}},maxiter,\text{tol}_\delta$}
\State{$x_1,\ldots,x_{2n+1} \leftarrow x_1^{\text{init}},\ldots,x_{2n+1}^{\text{init}}$}
\Comment{initial interpolation nodes}
\For{$k=1,\ldots,maxiter$}
\State{$r\leftarrow$ $(n,n)$-rational interpolant $r(\mathrm{i} x) \approx \mathrm{e}^{\mathrm{i}\omega x}$ at nodes $x_1,\ldots,x_{2n+1}$}
\State{$err(x) := |r(\mathrm{i} x)-\mathrm{e}^{\mathrm{i}\omega x}|$}
\State{$\eta_1,\ldots,\eta_{2n+2}\leftarrow find\_local\_error\_max(x_1,\ldots,x_{2n+1},err,\ldots)$}
\Comment{as in~\eqref{eq:localmaxerr}}%
\label{alg:line:findlocalmax}
\State{$\delta\leftarrow 1-\min_k err(\eta_k)/\max_j err(\eta_j)$}
\Comment{error in uniformity~\eqref{eq:derapproxerr}}
\State{$\phi_j\leftarrow \operatorname{angle}(r(\mathrm{i} \eta_j)/\mathrm{e}^{\mathrm{i} \omega\eta_j})$ for $j=1,\ldots,2n+2$}
\Comment{phase error~\eqref{eq:phaseerrangle}}
\If{$\phi_1,\ldots,\phi_{2n+2}$ are alternating, $\max_j err(\eta_j) < 2$ and $\delta<\text{tol}_\delta$}
\State{\Return{$r$}}
\EndIf
\State{$x_1,\ldots,x_{2n+1} \leftarrow interpolation\_nodes\_correction(\ldots)$}
\label{alg:line:nodecorrection}
\EndFor
\State{\Return{$r$}}
\end{algorithmic}
\end{algorithm}

\subsubsection{A symmetric setting}\label{subsec:brib:sym}
The unitary best approximant is symmetric~\cite[Section~6]{JS23u} and its interpolation nodes~\eqref{eq:urbinterpolate} are mirrored around zero as in~\eqref{eq:xjmirrored}. It seems natural to choose mirrored interpolation nodes when initializing the interpolation-based algorithm, which implies that the rational interpolant computed in the first iteration is symmetric due to Proposition~\ref{prop:symmetry}. Consequently this carries over to the points $\eta_1,\ldots,\eta_{2n+2}$ as in~\eqref{eq:etajmirrored}. 

The approaches for interpolation nodes correction suggested in Subsection~\ref{subsec:brib:nodecorrection} keep the interpolation nodes mirrored around zero as in~\eqref{eq:xjmirrored} in practice. Thus, we may assume a symmetric setting throughout the interpolation-based algorithm which can be utilized to reduce computational cost at various steps of the algorithm. 

\subsubsection{Initial interpolation nodes}
Performance of the presented algorithm depends on the set of initial interpolation nodes, i.e., $x_1^{\text{init}},\ldots,x_{2n+1}^{\text{init}}\!\!\in(-1,1)$ in \algorithmname~\ref{alg:bribiteration}. From~\cite[Proposition~9.3]{JS23u} we recall that in the limit $\omega\to 0$ the interpolation nodes $x_1,\ldots,x_{2n+1}$ converge to the Chebyshev points, i.e., the zeros of the Chebyshev polynomial of degree $2n+1$. On the other hand, in the limit $\omega\to (n+1)\pi$ the interpolation nodes converge to equispaced points, namely, $x_j\to-1 + j/(n + 1)$ for $j=1,\ldots,2n+1$ as shown in~\cite[Proposition~9.5]{JS23u}. We suggest using a linear combination of these two limits as initial interpolation nodes, i.e.,
\begin{equation}\label{eq:xinit}
x_j^{\text{init}} = (1-\xi)\theta_j + \xi (-1 + j/(n + 1)),~~~j=1,\ldots,2n+1.
\end{equation}
where $\xi=\omega/((n+1)\pi)$ and $\theta_1,\ldots,\theta_{2n+1}\in(-1,1)$ correspond to the Chebyshev points. 

We remark that for Algorithm~G~\cite{Fra76} the authors suggest using Chebyshev points as initial nodes to compute Chebyshev approximants to real functions. It was recently shown in~\cite{Ja24cheb} that Chebyshev approximants on shrinking domains attain interpolation nodes which approach scaled Chebyshev points in general settings. Thus, taking Chebyshev points as initial interpolation nodes for the problems considered in~\cite{Ma63, Fra76, Ho21} seems justified in certain settings.

\subsubsection{Restarting the iteration and recycling interpolation nodes }
The discussed strategies for interpolation nodes correction in \algorithmname~\ref{alg:bribiteration} depend only on a single iteration step. Thus, this algorithm can be restarted by providing the previously computed interpolation nodes as initial interpolation nodes $x_1^{\text{init}},\ldots,x_{2n+1}^{\text{init}}$ without delaying convergence.

Moreover, in case the unitary best approximant for a frequency $\omega_1$ is desired and the interpolation nodes of the unitary best approximant for a frequency $\omega_2\approx\omega_1$ are available, using these nodes as initial interpolation nodes when computing the unitary best approximant for $\omega_1$ usually decreases the number of required iterations.

\subsubsection{Finding points of intermediate maxima}\label{subsec:brib:findmax}
Various approaches to find local maxima of functions can be applied to find the intermediate points of maximal error $\eta_j$ in line~\ref{alg:line:findlocalmax} in \algorithmname~\ref{alg:bribiteration}, e.g., routines suggested in~\cite{Ma63,Ho21}.
Depending on the initial nodes and the choice of $n$ and $\omega$, at first iterations it might occur that the approximation error is below computer arithmetic in some subintervals or $\|r-\exp(\omega \cdot)\|=2$. In these cases, the intermediate maxima $\varepsilon_j$ can be approximately computed by sampling the approximation error at equispaced points on each subinterval which is usually sufficient for interpolation nodes correction in the initial phase as discussed in the following subsection.

In practice, finding the points $\eta_j$ or $\varepsilon_j$ at each iteration cycle usually occupies a major part of the runtime of \algorithmname~\ref{alg:bribiteration}. However, the choice of the subroutine to find the intermediate maxima plays a minor role considering convergence rates and overall success of computing the unitary best approximant in general, and is not discussed in full detail in the present work. Nevertheless, we suggest to take advantage of the symmetric setting to reduce computational cost for this procedure, particularly, the identity~\eqref{eq:etajmirrored} which halves the number of points to be localized.

\subsubsection{Interpolation nodes correction}\label{subsec:brib:nodecorrection}

In the present subsection we discuss strategies for interpolation nodes correction as used in line~\ref{alg:line:nodecorrection} of \algorithmname~\ref{alg:bribiteration}. In particular, we consider the strategies of Maehly's second method~\cite{Ma63}, Algorithm~G~\cite{Fra76} and the BRASIL algorithm~\cite{Ho21} which were originally introduced for interpolation nodes correction in a real setting. For the strategy of Maehly's second method we also consider a simplified version as suggested in~\cite{Du65}.

\paragraph{BRASIL algorithm}

We consider the strategy of the BRASIL algorithm as in~\cite[Algorithm~2]{Ho21} which re-scales the length of the subintervals between the interpolation nodes based on the intermediate maximal errors.
This strategy ensures that the interpolation nodes remain in the interval $(-1,1)$, and seems to provide the most robust choice for nodes correction. It eventually succeeds nodes correction for the case $\|r-\exp(\omega \cdot)\|=2$, or if the points $\eta_j$ are not properly detected in sub-intervals with local errors below computer precision, which may occur in the initial phase of the algorithm for large $n$.

The strategy of the BRASIL algorithm is sketched in \algorithmname~\ref{alg:BRASIL}. The scaling factor $\kappa$ therein can be understood as a scaled version of the scaling factor $\tau$ in~\cite[Algorithm~2]{Ho21}, namely $\tau = \kappa/n$. We suggest using $\kappa = 2.2$ for the scaling factor which seems particularly relevant for success of this strategy.

\begin{algorithm}
\caption{The strategy for interpolation nodes correction used in the BRASIL algorithm~\cite{Ho21}.}
\label{alg:BRASIL}
\begin{algorithmic}[1]
\Require{$x_1,\ldots,x_{2n+1}, \varepsilon_1, \ldots, \varepsilon_{2n+2}, \sigma_{\max},\kappa$}
\Comment{$\varepsilon_j = |r(\mathrm{i} \eta_j)-\mathrm{e}^{\mathrm{i}\omega \eta_j}|$}
\State{$\overline{\varepsilon} = \sum_{j=1}^{2n+2} \varepsilon_j/(2n+2)$}
\State{$\overline{\gamma} = \max_{j=1,\ldots,2n+2} |\varepsilon_j-\overline{\varepsilon} |$}
\State{$\sigma =\min\{\sigma_{\max},\kappa\overline{\gamma}/(n\overline{\varepsilon}) \}$}
\Comment{scaling factor $\kappa/n$}
\State{$\gamma_k =(\varepsilon_k-\overline{\varepsilon} )/\overline{\gamma},~~~k=1,\ldots,2n+2$}
\State{$\ell_k = (1-\sigma)^{\gamma_k}(x_k-x_{k-1}),~~~k=1,\ldots,2n+2$}
\Comment{using $x_{0}=-1$ and $x_{2n+2}=1$}
\State{$x_j \leftarrow \sum_{k=1}^{j}\ell_k \big/ \sum_{m=1}^{2n+2} \ell_m ,~~~j=1,\ldots,2n+1$}
\State{\Return{$x_1,\ldots,x_{2n+2}$}}
\end{algorithmic}
\end{algorithm}

\paragraph{Maehly's second method}
We proceed to apply the ideas described in~\cite[Section~9]{Ma63} and~\cite{Du65} for interpolation nodes correction. 
Consider the setting~\eqref{eq:nonuniforminterpsetting} were $x_1,\ldots,x_{2n+1}$ and $\eta_1,\ldots,\eta_{2n+2}$ correspond to the current interpolation nodes and points of intermediate maxima, respectively, and let $\varepsilon_1,\ldots,\varepsilon_{2n+2}$ denote the intermediate maximal errors as in~\eqref{eq:epsj}.
A correction $\delta x_k$ for the interpolation node $x_k$ is described in~\cite[eq.~(9.5)]{Ma63} via the linear system
\begin{equation*}
\log \lambda + \sum_{k=1}^{2n+1} \frac{\delta x_k}{\eta_j - x_k }
= \log \varepsilon_j,~~~j=1,\ldots,2n+2,
\end{equation*}
where $\lambda>0$ and $\delta x_1,\ldots,\delta x_{2n+1}$ are unknowns. By eliminating $\lambda$, the correction step corresponds to
\begin{equation}\label{eq:Ma63correctionxsym}
x_j \leftarrow x_j + \delta x_j,~~~j=1,\ldots,2n+1,~~\text{where $\delta x=(\delta x_j)_{j=1}^{2n+1}$ solves $M\delta x=b$}.
\end{equation}
for the matrix $M\in\mathbb{R}^{(2n+1)\times(2n+1)}$ defined by
\begin{equation}\label{eq:Ma63correctionM}
M_{jk} = \frac{\eta_1-\eta_{j+1}}{(\eta_{j+1}-x_k)(\eta_1-x_k)},~~~j,k=1,\ldots,2n+1,
\end{equation}
and the right-hand side vector $b\in\mathbb{R}^{2n+1}$ defined by
\begin{subequations}
\begin{equation}\label{eq:Ma63correctionbj1}
b_j = \log(\varepsilon_{j+1}/\varepsilon_1),~~~j=1,\ldots,2n+1.
\end{equation}
This system is described as well-conditioned which shows to be true for the examples tested in the present work. However, we only suggest using this strategy in case that the phase error has an alternating sign at the points $\eta_1,\ldots,\eta_{2n+2}$ (similar to~\eqref{eq:phaseerrnonunfaltern}), the points $\eta_j$ are available with sufficient precision and the rational interpolant satisfies $\|r-\exp(\omega \cdot)\|<2$, conditions which might not hold true at an initial phase. In this case, the correction of the interpolation nodes in~\eqref{eq:Ma63correctionxsym} could fail, unlike the strategy of the BRASIL algorithm, to provide interpolation nodes in ascending order inside the interval $(-1,1)$.

Assuming that the alternating error is alternating, it seems to further improve stability of the correction strategy to modify the right-hand side of the system in line with~\cite[eq.~(9.6)]{Ma63} by replacing the vector $b$ by an approximation based on a bilinear transform, i.e.,
\begin{equation}\label{eq:Ma63correctionbj2}
b_j \approx 2(\varepsilon_{j+1}-\varepsilon_1)/(\varepsilon_{j+1}+\varepsilon_1),~~~j=1,\ldots,2n+1.
\end{equation}
\end{subequations}
We suggest using~\eqref{eq:Ma63correctionbj1} for $\delta<0.1$ and the approximation~\eqref{eq:Ma63correctionbj2} otherwise, since the former shows better performance but seems to be more prone to return invalid interpolation nodes for larger $\delta$.

A simplification of Maehly's strategy is suggested in~\cite{Du65}. In particular, the authors apply a direct formula~\cite[eq.~(5)]{Du65} (based on a direct formula in~\cite{Sche59}) for the solution of the related linear system which improves performance of the algorithm and decreases computational cost.
Utilizing the ideas of \cite{Du65} the following formula can be used to evaluate the correction $\delta x_j$ in~\eqref{eq:Ma63correctionxsym} with $b_\ell$ as defined below. Namely, 
\begin{equation}\label{eq:MDdeltaxdirect}
\delta x_j = \frac{\prod_{k=1}^{2n+2} (x_j-\eta_k)}{\prod_{\substack{k=1 \\ k\neq j}}^{2n+1} (x_j-x_k)}
\sum_{\ell = 1}^{2n+2} \left(\frac{b_\ell}{x_j-\eta_\ell} \frac{\prod_{k=1}^{2n+1} (\eta_\ell-x_k)}{\prod_{\substack{k=1 \\ k\neq \ell}}^{2n+2} (\eta_\ell-\eta_k)}\right),~~~j=1,\ldots,2n+1,
\end{equation}
with
\begin{subequations}
\begin{equation}\label{eq:Ma63correctionbj1x}
b_j = \log(\varepsilon_j/\overline{\varepsilon}),~~~j=1,\ldots,2n+2,
\end{equation}
where $\overline{\varepsilon}$ refers to the geometric mean of $\varepsilon_1,\ldots,\varepsilon_{2n+2}$. The formula in~\eqref{eq:MDdeltaxdirect} can be evaluated with computational cost of $\mathcal{O}(n^2)$, comparing with cost of $\mathcal{O}(n^3)$ for solving the related linear equation, cf.~\cite{Du65}.
Similar as above, we may also consider the approximation
\begin{equation}\label{eq:Ma63correctionbj2x}
b_j \approx 2 (\varepsilon_j-\overline{\varepsilon})/(\varepsilon_j+\overline{\varepsilon}),~~~j=1,\ldots,2n+2,
\end{equation}
\end{subequations}
in case of $\delta>0.1$ to modify the equation.

We remark that symmetry properties can be used to further simplify the interpolation nodes correction~\eqref{eq:Ma63correctionxsym} as well as the simplified formula~\eqref{eq:MDdeltaxdirect}.

\paragraph{Algorithm G}
The strategy for interpolation nodes correction suggested in~\cite{Fra76} (particularly in Section~4 therein) consists of a time propagation step for the system~\cite[eq.~(1)]{Fra76} using Euler's method. This strategy is introduced as an improvement to the strategy in Maehly's second method in terms of stability to compute approximants to certain functions as noted therein. However, for interpolation nodes correction for the unitary best approximant this strategy shows to perform worse than the other strategies and is not further considered in the present work.

\subsubsection{Recommended strategy}\label{subsec:bestalgorithm}
We suggest combining the strategy of interpolation nodes correction of the BRASIL method and Maehly's second method as follows. In case that points of maximal errors at some sub-intervals are not properly detect, e.g., due to limitations of computer arithmetic, or that the phase error is not alternating, we suggest using the strategy of the BRASIL algorithm due to its robustness in such settings. Such cases might occur at the first iterations for certain $n$ and $\omega$.
Provided the unitary best approximant for the given $n$ and $\omega$ has an approximation error sufficiently above computer arithmetic which is certainly the case if $\omega$ satisfies~\eqref{eq:problemfindw} for a respective error level $\varepsilon$, the error of the computed approximant is eventually above computer arithmetic after some correction cycles using the strategy of the BRASIL algorithm. Usually, $\|r-\exp(\omega\cdot)\|<2$ holds true at that point as well and the phase error has an alternating sign at the points $\eta_1,\ldots,\eta_{2n+2}$.
Once these conditions hold true, we suggest to apply the correction strategy of Maehly's second method to increase convergence speed, using the simplified formula~\eqref{eq:MDdeltaxdirect}. Initially, using the modified right-hand sides for the linear systems, e.g.,~\eqref{eq:Ma63correctionbj2x}, until the error in uniformity satisfies $\delta<0.1$. Once this condition holds true, using the strategy of Maehly's second method without modification, e.g.,~\eqref{eq:Ma63correctionbj1x}, further improves convergence and can be used until the stopping criterion is met.

\subsection{AAA-Lawson method}\label{subsec:alglawson}
In the present subsection we consider the AAA-Lawson method~\cite{NST18,NT20}, which is also listed as approach~\ref{item:alglawson} further above. The AAA-Lawson method consists of an initial AAA phase in which a rational approximant is constructed by rational interpolation in nodes adaptively selected from a given set of test nodes, and a Lawson-type iteration which further improves accuracy of the approximant minimizing a successively reweighted least-squares problem.

The AAA method was originally designed to construct near-best approximants to real as well as complex functions, and the AAA-Lawson method aims to construct best approximants in a Chebyshev sense for such problems. Thus, these approaches can be directly applied to construct $(n,n)$-rational approximants to~$\mathrm{e}^{\mathrm{i} \omega x}$. Moreover, it was recently shown in~\cite{JS24} that approximants to $\mathrm{e}^{\mathrm{i}\omega x}$ generated by the AAA and AAA-Lawson methods are unitary. Thus, these methods provide good candidates for the unitary best approximant.

The error in uniformity $\delta$ is defined in~\eqref{eq:derapproxerr} for an interpolatory setting which is not necessary given for the approximant computed by the AAA-Lawson method. However, in case the computed approximant satisfies the respective interpolation condition the results of Corollary~\ref{cor:errsandwich} can potentially be used to evaluate performance of the AAA-Lawson method on the run. Compared to the interpolation-based algorithm this requires additional computational effort since interpolation nodes $x_j$ and points of intermediate maximal errors $\eta_j$ have to be detected first which might not be practical.

For the numerical experiments done for the AAA and AAA-Lawson methods we choose equispaced test nodes, and test nodes which are adjusted on the run as suggested in~\cite{DNT24}. Numerical experiments are provided in Section~\ref{sec:convergence}. The figures therein also illustrate the error in uniformity for the AAA-Lawson method in case the respective interpolation condition holds true.

\section{Numerical experiments and conclusion}\label{sec:convergence}
In the present section we illustrate accuracy of the a~priori estimates introduced in Section~\ref{sec:setting} and performance of the algorithms from Section~\ref{sec:algorithms} to compute the unitary best approximant. For the latter we consider the approach \ref{item:algbrib} using the strategy of the BRASIL algorithm and the combined strategy suggested in Subsection~\ref{subsec:bestalgorithm} which mostly corresponds to the strategy of Maehly's second method. Moreover, the AAA and AAA-Lawson methods are tested with equispaced and adaptively chosen test nodes as considered in Subsection~\ref{subsec:alglawson}.

{\bf A~priori estimates for $\omega$.} 
In \figurename~\ref{fig:testwest} we consider the estimates from Section~\ref{sec:setting} to determine $\omega$ s.t.~\eqref{eq:problemfindw} holds true for different degrees $n$ from $n=2$ to $n=1024$ and error objectives $\varepsilon$ between $\varepsilon=10^{-12}$ and $\varepsilon=0.3$.
Namely, the estimates $\omega_e$~\eqref{eq:est2w} and $\omega_a$~\eqref{eq:asymestw} discussed therein. The choices of $n$ and $\varepsilon$ aim to cover a wide range of practical degrees and error objectives which are mostly distinct to the degrees and error objectives used to generate the estimate $\omega_e$. For each tested degree $n$ and error objective $\varepsilon$ the errors attained by the unitary best approximant computed for the frequencies $\omega_e(n,\varepsilon)$ and $\omega_a(n,\varepsilon)$ are plotted over $n$. Thus, this figure illustrates how the actual errors for these frequencies match the error objective $\varepsilon$ which is also illustrated for each tested value therein. We conclude that for small $n$, approximately $n\leq 20$, and small or moderate error objectives $\varepsilon$ the asymptotic error estimate $\omega_a$ performs sufficiently well. For larger $n$ or $\varepsilon$ we suggest using the estimate $\omega_e$ based on experimental data. To implement this observation in practice for a given $n$ we suggest using the estimate $\omega_a$ for $\varepsilon < 10^{-2(n-4)/3}$, and $\omega_e$ otherwise.

While the experimental data to construct the estimate $\omega_e$ was computed using higher precision arithmetic, we used standard double precision arithmetic to compute the results for $\omega_e$ and $\omega_a$ in \figurename~\ref{fig:testwest}. For some larger values of $n$ the accuracy of the approximants for $\omega_e(n,\varepsilon)$ with $\varepsilon\approx 10^{-12}$ is limited by computer arithmetic to some degree.

\begin{figure}
    \centering
    \includegraphics[width=0.9\textwidth]{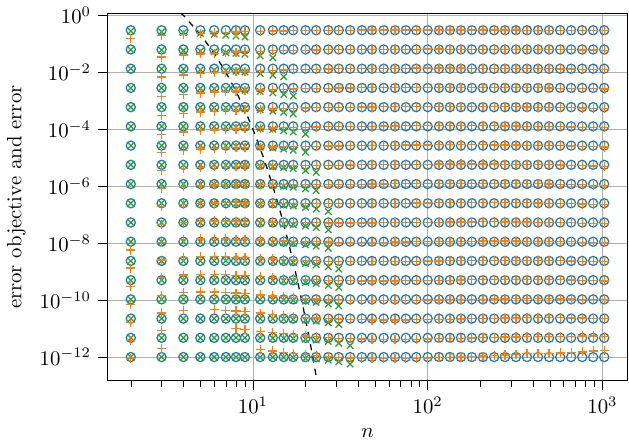}
    \caption{The symbols '$\circ$' mark a set of error objectives $\varepsilon$ over degrees $n$ which are used to construct estimates for $\omega$ with the aim that the respective unitary best approximant attains an error of $\varepsilon$. The errors of the computed unitary best approximants using the estimates $\omega_e(n,\varepsilon)$~\eqref{eq:est2w} and $\omega_a(n,\varepsilon)$~\eqref{eq:asymestw} for $\omega$ are marked by '$+$' and '$\times$', respectively. Marks close to each other correspond to the same value of $\varepsilon$, and marks are neglected in case the error computed for $\omega_e$ or $\omega_a$ is too far from the error objective. The dashed line shows $10^{-2(n-4)/3}$ over $n$.}
    \label{fig:testwest}
\end{figure}

{\bf Approximation error and error in uniformity in computer arithmetic.}
Figures~\ref{fig:n32brib}--\ref{fig:n256AAAL} show the error $\|r-\exp(\omega\cdot)\|$ and the error in uniformity~$\delta$ (defined in~\eqref{eq:derapproxerr} for an interpolatory setting) for computed approximants. While $\delta$ is well-defined for approximants computed by the interpolation-based algorithm, this is only the case for the AAA-Lawson method if the respective interpolation condition is met. In particular, this holds true when approximants are sufficiently close to the unitary best approximant. 

The error in uniformity~$\delta$ provides a strong tool to illustrate performance of algorithms due to Corollary~\ref{cor:errsandwich}. Besides an upper bound for the relative distance between the errors of the computed and the unitary best approximant, this corollary shows that $\delta\to 0$ implies convergence of the computed approximant to the unitary best approximant.
Making use of the notation $\varepsilon_j$ and~\eqref{eq:errnormismaxetaj} we note that the error in uniformity $\delta$ corresponds to
\begin{equation*}
\delta
= \left.\left(\max_{k=1,\ldots,2n+2} \varepsilon_k - \min_{j=1,\ldots,2n+2} \varepsilon_j\right) \right/  \|r-\exp(\omega\cdot)\|.
\end{equation*}
In practice, the accuracy when computing $\delta$ is approximately limited by computer precision divided by the uniform error due to limitations of computer arithmetic when evaluating $\varepsilon_1,\ldots,\varepsilon_{2n+2}$. Thus, $\delta\to0$ is only observed roughly up to this accuracy.
E.g., in our numerical experiments which illustrate results for the interpolation-based algorithm using double precision arithmetic the error in uniformity starts to stagnate at $ \delta \approx 10^{-14} / \|r-\exp(\omega\cdot)\| $ for $n=32$ (\figurename~\ref{fig:n32brib}) and $ \delta \approx 10^{-12} / \|r-\exp(\omega\cdot)\| $ for $n=256$ (\figurename~\ref{fig:n256brib}).
In our numerical experiments with the AAA-Lawson method a similar behavior would be expected. However, the approximants computed by the AAA and AAA-Lawson methods show significantly larger errors in uniformity.

{\bf Performance of algorithms.}
To compare performance of different algorithms we choose the degrees $n=32$ and $n=256$ for different choices of $\omega$ s.t.\ the unitary best approximant approximately attains the error levels $\varepsilon=10^{-12},10^{-10},\ldots,10^{-2},10^{-1}$. As a reference, the respective frequencies $\omega$ used for numerical examples are illustrated Table~\ref{table:wanderror}. The errors shown in this table correspond to the error of the approximant computed by the interpolation-based algorithm with an error in uniformity of $\delta<10^{-6}$ using higher precision arithmetic.

\begin{table}
\caption{Values of $\omega$ used to test performance of algorithms for $n=32$ and $n=256$. The frequencies $\omega$ are chosen s.t.\ the unitary best approximant for $\omega$ and $n$ attains certain error levels, and are rounded to two decimal places for reusability. For each $\omega$ we show the reference error which corresponds to the error $\|r-\exp(\omega\cdot)\|$ of the approximant computed using the interpolation-based algorithm with higher precision arithmetic and error in uniformity $\delta<10^{-6}$.}
\label{table:wanderror}
\begin{tabular}{|c|c||c|c|}
\hline
\multicolumn{2}{|c||}{$n=32$}&\multicolumn{2}{c|}{$n=256$}\\\hline
$\omega$& ref.\ error $\approx$&
$\omega$& ref.\ error $\approx$\\\hline
$95.48$&$1.00\cdot10^{-1}$ &$797.18$&$1.00\cdot10^{-1}$\\
$91.35$&$1.00\cdot10^{-2}$ &$791.45$&$1.00\cdot10^{-2}$\\
$84.16$&$1.00\cdot10^{-4}$ &$780.93$&$1.00\cdot10^{-4}$\\
$77.86$&$1.01\cdot10^{-6}$ &$771.16$&$1.00\cdot10^{-6}$\\
$72.19$&$1.01\cdot10^{-8}$ &$761.89$&$1.00\cdot10^{-8}$\\
$67.03$&$1.01\cdot10^{-10}$&$753.01$&$1.01\cdot10^{-10}$\\
$62.29$&$1.00\cdot10^{-12}$&$744.44$&$1.00\cdot10^{-12}$\\
\hline
\end{tabular}
\centering
\end{table}

In figures~\ref{fig:n32brib} and~\ref{fig:n256brib} we illustrate the performance of the interpolation-based algorithm for $n=32$ and $n=256$. We show results for the frequencies $\omega$ provided in Table~\ref{table:wanderror}. In these figures the error $\|r-\exp(\omega\cdot)\|$ (top row) and the error in uniformity $\delta$ (bottom row) are plotted over the iteration number (cf.~iteration in \algorithmname~\ref{alg:bribiteration}) for the interpolation nodes correction strategy of the BRASIL algorithm, and the combined strategy suggested in Subsection~\ref{subsec:bestalgorithm}, which for these example only consists of the strategy of Maehly's second method. The combined strategy succeeds in reaching a smallest possible error in uniformity within a small number of iterations, i.e., approximately computer precision divided by the respective error $\|r-\exp(\omega\cdot)\|$.

\begin{figure}
    \centering
    \includegraphics[width=0.9\textwidth]{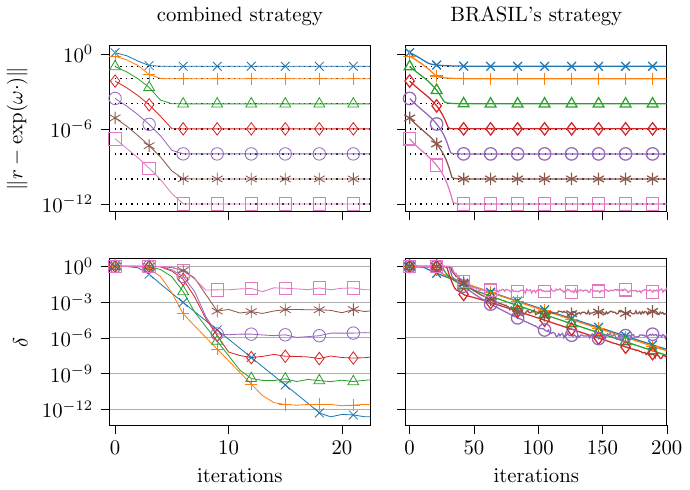}
    \caption{These plots illustrate the performance of the interpolation-based algorithm for $n=32$ and different frequencies $\omega$ over the number of iterations. The top and bottom rows show the error $\|r-\exp(\omega\cdot)\|$ and the error in uniformity $\delta$, respectively, for the interpolation-based algorithm using the combined strategy (left column) and the strategy of the BRASIL algorithm (right column) for interpolation nodes correction. The combined strategy is introduced in Subsection~\ref{subsec:bestalgorithm}, and for the present example, only consists of the strategy of Maehly's second method. In each plot different lines show results for the different frequencies $\omega$ provided in Table~\ref{table:wanderror}, i.e., $\omega = 95.48$ '$\times$', $91.35$ '$+$', $84.16$ '$\Delta$', $77.86$ '$\Diamond$', $72.19$ '$\circ$', $67.03$ '$*$', and $62.29$ '$\square$'. The reference errors from this table are also illustrated in the plots in the top row as dotted horizontal lines.
}
    \label{fig:n32brib}
\end{figure}
\begin{figure}
    \centering
    \includegraphics[width=0.9\textwidth]{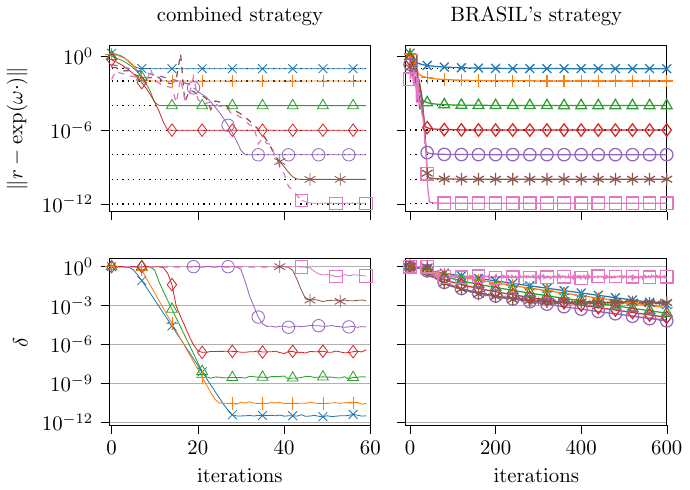}
    \caption{
    These plots show the error $\|r-\exp(\omega\cdot)\|$ and the error in uniformity $\delta$ of the interpolation-based algorithm over the number of iterations using different strategies for interpolation nodes correction similar to \figurename~\ref{fig:n32brib}. The present figure shows results for $n=256$ and $\omega = 797.18$ '$\times$', $791.45$ '$+$', $780.93$ '$\Delta$', $771.16$ '$\Diamond$', $761.89$ '$\circ$', $753.01$ '$*$', and $744.44$ '$\square$' as provided in Table~\ref{table:wanderror}. For $\omega=761.89$, $753.01$ and $744.44$ the combined strategy applies BRASIL's strategy in an initial phase, and for these iterations the errors are illustrated by dashed lines without marks in the plots. Otherwise, the combined strategy applies the strategy of Maehly's second method for the present examples.
    }
    \label{fig:n256brib}
\end{figure}

In figures~\ref{fig:n32AAAL} and~\ref{fig:n256AAAL} we illustrate the performance of the AAA-Lawson method for $n=32$ and $n=256$. The error (top row) and the error in uniformity (bottom row) are plotted over the number of Lawson iterations. The three columns in each figure correspond to data for adaptively chosen test nodes using the implementation of~\cite{DNT24} (left column), and different numbers of equispaced test nodes (middle \& right columns). For the latter we choose a smaller and larger number of equispaced test nodes s.t.\ the AAA-Lawson method shows convergence for most $\omega$. In particular, for $n=32$ we choose $1100$ and $4900$ equispaced test nodes on $[-1,1]$ whereas the implementation of~\cite{DNT24} is using $640$ test nodes. For $n=256$ we choose $6000$ and $35\,000$ equispaced test nodes and the implementation of~\cite{DNT24} is using $5120$ test nodes. We remark that the performance of the AAA-Lawson method improves with the number of test nodes to some point but not consistently, and might be limited by the discretization of $\mathrm{e}^{\mathrm{i}\omega x}$ at the test nodes. In particular, considering the error in uniformity the AAA-Lawson method doesn't converge as well as for the interpolation-based algorithm for the presented examples. However, the approximation error attained by the AAA-Lawson method is certainly on the level of the reference error within a small number of Lawson iterations for most $\omega$.

Approximants computed by the AAA-Lawson method might not satisfy the necessary interpolation conditions to define the error in uniformity $\delta$ and in such cases no values are shown for $\delta$ in the corresponding plots. Moreover, for some choices of $n$ and $\omega$ the approximants computed by the AAA-Lawson method fail to attain a non-maximal approximation error in which case no results are shown for the respective error plots.

\begin{figure}
    \centering
    \includegraphics[width=0.9\textwidth]{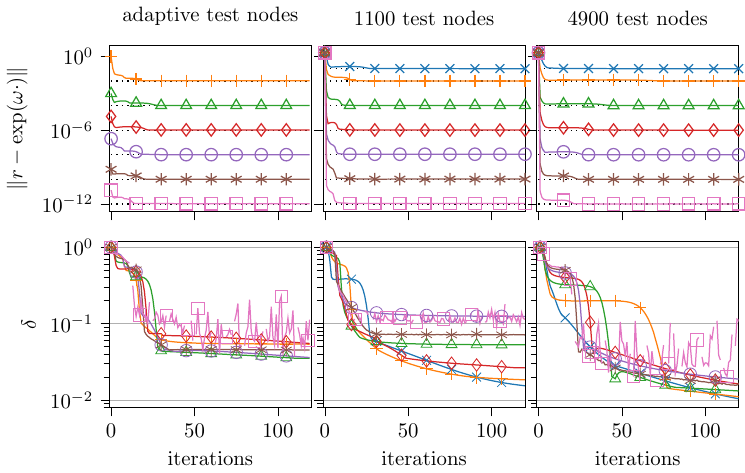}
    \caption{
    These plots illustrate the performance of the AAA-Lawson method using different sets of test nodes for $n=32$ and different frequencies $\omega$ over the number of Lawson iterations. The top and bottom rows show the error $\|r-\exp(\omega\cdot)\|$ and the error in uniformity $\delta$, respectively, for test nodes which are adjusted on the run (left column), $1100$ equispaced test nodes (middle column) and $4900$ equispaced test nodes (right column). In each plot different lines show results for the different frequencies $\omega$ provided in Table~\ref{table:wanderror}, i.e., $\omega = 95.48$ '$\times$', $91.35$ '$+$', $84.16$ '$\Delta$', $77.86$ '$\Diamond$', $72.19$ '$\circ$', $67.03$ '$*$', and $62.29$ '$\square$'. The reference errors from this table are also illustrated in the plots in the top row as dotted horizontal lines.
    No results are shown for the error in uniformity in case the computed approximant does not attain $2n+1$ interpolation nodes. Moreover, no results are shown for the errors in case the computed approximant fails to approach a non-maximal error.}
    \label{fig:n32AAAL}
\end{figure}

\begin{figure}
    \centering
    \includegraphics[width=0.9\textwidth]{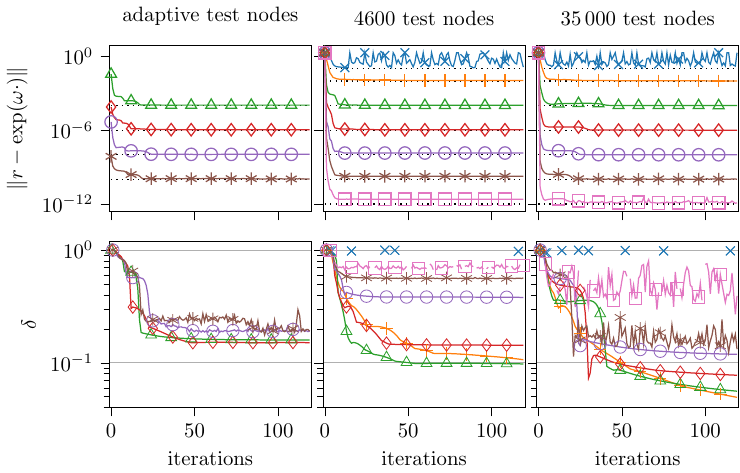}
    \caption{
    These plots show error $\|r-\exp(\omega\cdot)\|$ and error in uniformity $\delta$ of approximants constructed by the AAA-Lawson method using test nodes which are adjusted on the run (left column), $4600$ equispaced test nodes (middle column) and $35\,000$ equispaced test nodes (right column) over the number of Lawson iterations similar to \figurename~\ref{fig:n32AAAL}.  The present figure shows results for $n=256$ and $\omega = 797.18$ '$\times$', $791.45$ '$+$', $780.93$ '$\Delta$', $771.16$ '$\Diamond$', $761.89$ '$\circ$', $753.01$ '$*$', and $744.44$ '$\square$' as provided in Table~\ref{table:wanderror}.}
    \label{fig:n256AAAL}
\end{figure}

Performance of the AAA method, without Lawson iteration, is evaluated in Table~\ref{table:errorsAAA} for degrees $n=32$ and $n=256$, using adaptively chosen test nodes as well as equispaced test nodes ($4900$ and $35\,000$ equispaced test nodes for $n=32$ and $n=5200$, respectively).

\begin{table}
\caption{Error of approximants computed by the AAA method with equispaced test nodes and adaptively chosen test nodes are labeled as 'AAA' and 'aAAA', respectively, in this table. For 'AAA' we use $4900$ and $35\,000$ equispaced test nodes for $n=32$ and $n=256$, respectively.}
\label{table:errorsAAA}
\begin{tabular}{|c|c|c||c|c|c|}
\hline
\multicolumn{3}{|c||}{$n=32$}&\multicolumn{3}{c|}{$n=256$}\\\hline
$\omega$&AAA $\approx$&aAAA $\approx$&
$\omega$&AAA $\approx$&aAAA $\approx$\\\hline
$95.48$&$1.32$&$2.00$
& $797.18$&$2.00$&$2.00$
\\ $91.35$&$1.17\cdot10^{-1}$&$1.22$
& $791.45$&$5.37\cdot10^{-1}$&$2.00$
\\ $84.16$&$5.47\cdot10^{-4}$&$1.21\cdot10^{-3}$
& $780.93$&$4.35\cdot 10^{-3}$&$3.37\cdot 10^{-2}$
\\ $77.86$&$3.05\cdot10^{-5}$&$1.63\cdot10^{-5}$
& $771.16$&$5.31\cdot 10^{-5}$&$6.69\cdot 10^{-5}$
\\ $72.19$&$2.09\cdot10^{-7}$&$1.55\cdot10^{-7}$
& $761.89$&$3.89\cdot 10^{-7}$&$4.33\cdot 10^{-6}$
\\ $67.03$&$2.12\cdot10^{-9}$&$7.41\cdot10^{-10}$
& $753.01$&$3.11\cdot 10^{-9}$&$1.01\cdot10^{-8}$
\\ $62.29$&$1.07\cdot10^{-11}$&$1.16\cdot10^{-11}$
& $744.44$& $2.33\cdot 10^{-11}$&$2.69\cdot 10^{-10}$
\\
\hline
\end{tabular}
\centering
\end{table}

{\bf Conclusions.}
The interpolation-based algorithm shows to approach the reference error consistently, which is also the case for the AAA-Lawson method with a few exceptions. While both approaches can be used to compute approximants of a certain accuracy for small and moderate degrees $n$, we remark that the interpolation-based algorithm succeeds to do so even for large degrees such as $n=1024$ while the AAA-Lawson method fails to compute an approximant in reasonable time in this case.

Considering the error in uniformity, the interpolation-based algorithm outperforms the AAA-Lawson method. In particular, for the latter the error in uniformity shows to be stagnating at moderate levels for the presented examples while for the interpolation-based algorithm the error in uniformity approaches the smallest values feasible for the underlying computer arithmetic. In particular, the combined strategy for nodes correction suggested in Subsection~\ref{subsec:bestalgorithm} achieves this level of precision within a small number of iterations.
Thus, while the interpolation-based algorithm and the AAA-Lawson method are both suitable to construct approximants of certain error levels, the former is clearly favorable to compute the unitary best approximant with higher requirements on equioscillatory properties.

\end{document}